\documentclass[a4paper,reqno,10pt]{amsart}
\usepackage[utf8]{inputenc}
\usepackage{amsfonts}
\usepackage{amsmath}
\usepackage{amsthm}
\usepackage{amssymb}
\usepackage{graphicx}
\usepackage{comment}
\usepackage[dvipsnames]{xcolor}

\usepackage[linktocpage,bookmarksopen=true]{hyperref}

\newtheoremstyle{teoremas}
{12pt}
{13pt}
{\itshape}
{}
{\bfseries}
{}
{.5em}
{}

\theoremstyle{teoremas}
\newtheorem{theorem}{Theorem}[section]
\newtheorem{thm}[theorem]{Theorem}
\newtheorem{corollary}[theorem]{Corollary}
\newtheorem{cor}[theorem]{Corollary}

\newtheorem{proposition}[theorem]{Proposition}
\newtheorem{prop}[theorem]{Proposition}
\numberwithin{equation}{section}

\newtheoremstyle{definition}
{12pt}
{12pt}
{}
{}
{\bfseries}
{}
{.5em}
{}

\theoremstyle{definition}
\newtheorem{definition}[theorem]{Definition}
\newtheorem{defn}[theorem]{Definition}
\newtheorem{conjecture}[theorem]{Conjecture}
\newtheorem{conj}[theorem]{Conjecture}

\newtheorem{problem}[theorem]{Problem}

\newtheorem{Q}[theorem]{Question}
\newtheorem{example}[theorem]{Example}
\newtheorem{ex}[theorem]{Example}
\newtheorem{remark}[theorem]{Remark}
\newtheorem{rem}[theorem]{Remark}

\newcommand{\M}{\mathsf{M}}

\newcommand{\PP}{\operatorname{PP}}
\newcommand{\CPP}{\operatorname{CPP}}
\newcommand{\ShPP}{\operatorname{ShPP}}
\newcommand{\SP}{\mathscr{SP}}
\newcommand{\RW}{\operatorname{RW}}

\usepackage{tikz}
\usetikzlibrary{automata, arrows}
\usepackage[scr=boondoxo]{mathalfa}

\usepackage[shortlabels]{enumitem} 

\hypersetup{
    colorlinks = true,
    linkbordercolor = {white},
    linkcolor = {NavyBlue},
    anchorcolor = {black},
    citecolor = {NavyBlue},
    filecolor = {cyan},
    menucolor = {NavyBlue},
    runcolor = {cyan},
    urlcolor = {NavyBlue}
}

\usepackage[margin=1.25in]{geometry} 
\DeclareMathOperator{\bsSchu}{\overleftarrow{\mathfrak{S}}}
\DeclareMathOperator{\Schu}{\mathfrak{S}}
\DeclareMathOperator{\cptn}{cptn}

\newcommand{\ehr}{\operatorname{ehr}}
\newcommand{\vol}{\operatorname{vol}}

\renewcommand{\emptyset}{\varnothing}
\def\la{\lambda}
\def\O{\mathcal{O}}

\usepackage{graphicx} 

\title{Skew shapes, Ehrhart positivity and beyond}

\author{Luis Ferroni, Alejandro H. Morales, Greta Panova}

\address{(L. Ferroni)
 Universit\`a di Pisa, Pisa, Italy
}
\email{luis.ferroni@unipi.it}
\address{(A. H. Morales)
Universit\'e du Qu\'ebec \`a Montr\'eal, Montr\'eal, Canada}
\email{morales\_borrero.alejandro@uqam.ca}

\address{(G. Panova)
 University of Southern California, Los Angeles (CA), United States
}
\email{gpanova@usc.edu}


\thanks{LF was a member at the Institute for Advanced Study, supported by the Minerva  Research Foundation, and is a members of the INDAM research group GNSAGA – Gruppo Nazionale per le
Strutture Algebriche, Geometriche e le loro Applicazioni. AHM is a member (Spring 2025) of the Institute of Advanced Study, supported by an anonymous donor and is partially supported by an NSF grant DMS-2154019 and an NSERC Discovery grant RGPIN-2024-06246. GP is partially supported by an NSF grant CCF-2302174 and a Simons Fellowship.}

\begin{document}
\allowdisplaybreaks

\begin{abstract}
    A classical result by Kreweras (1965) allows one to compute the number of plane partitions of a given skew shape and bounded parts as certain determinants. We prove that these determinants expand as polynomials with nonnegative coefficients. This result can be reformulated in terms of order polynomials of cell posets of skew shapes, and explains important positivity phenomena about the Ehrhart polynomials of shard polytopes, matroids, and order polytopes. Among other applications, we generalize a positivity statement from Schubert calculus by Fomin and Kirillov (1997) from straight shapes to skew shapes. We show that all shard polytopes are Ehrhart positive and, stronger, that all fence posets, including the zig-zag poset, and all circular fence posets have order polynomials with nonnegative coefficients. We discuss a general method for proving positivity which reduces to showing positivity of the linear terms of the order polynomials. We propose positivity conjectures on other relevant classes of posets.
\end{abstract}


\maketitle


\section{Introduction}

\subsection{Overview}

Let $P$ be a finite partially ordered set. In this article we shall be concerned with one of the most prominent polynomial invariants of $P$, called the \emph{order polynomial}. Throughout, this polynomial will be denoted by $\Omega(P;t)\in \mathbb{Q}[t]$. The order polynomial encodes useful information about the structure of $P$, including the number of linear extensions, the number of antichains, among many other statistics and features of relevance in mathematics. 

One geometric way of viewing the order polynomial of a poset $P$ is through the lens of the Ehrhart theory of polytopes. The \emph{Ehrhart polynomial} of a $d$-dimensional lattice polytope $\mathscr{P}\subseteq \mathbb{R}^n$ is the unique polynomial $\ehr(\mathscr{P},t)\in \mathbb{Q}[t]$ such that
    \[ \ehr(\mathscr{P},t) = \#(t\mathscr{P} \cap \mathbb{Z}^n)\]
for each positive integer $t$. The polynomiality of the map $t\mapsto \ehr(\mathscr{P},t)$ is a foundational result by Ehrhart~\cite{ehrhart}. 

The relationship between order polynomials and Ehrhart polynomials can be made explicit as follows: starting from a poset $P$ on $n$ elements, one can construct the so-called \emph{order polytope} $\mathscr{O}(P)\subseteq\mathbb{R}^n$ (for the precise definition see equation~\eqref{eq:order-polytope} below). Then, the following relation holds:
    \begin{equation} \label{eq:order-ehrhart}
    \ehr(\mathscr{O}(P), t) = \Omega(P;t+1).
    \end{equation}
    
For general lattice polytopes $\mathscr{P}$, it can be proved that the degree of the Ehrhart polynomial is $d=\dim \mathscr{P}$ and that it has constant term equal to $1$. Moreover, writing
    \[ \ehr(\mathscr{P},t) = a_d\, t^d + a_{d-1}\, t^{d-1} + \cdots + a_1 t + 1,\]
then one has that $a_d = \vol(\mathscr{P})$, and $a_{d-1} = \frac{1}{2}\vol(\partial\mathscr{P})$, where $\vol(\mathscr{P})$ stands for the relative volume of the polytope $\mathscr{P}$ with respect to intersection of the lattice $\mathbb{Z}^n$ with the affine span of $\mathscr{P}$ and, similarly, $\vol(\partial\mathscr{P})$ stands for the sum of the relative volumes of the facets of $\mathscr{P}$. In light of this, the two highest degree coefficients and the constant one are nonnegative for all lattice polytopes $\mathscr{P}$. A technical result due to McMullen \cite{mcmullen} provides a family of complicated formulas for the `middle coefficients', i.e., those of degree $2,\ldots, d-2$, in terms of volumes of faces of $\mathscr{P}$. Notwithstanding these formulas, unfortunately the middle coefficients can be negative---and more so, all of them can be negative at the same time \cite{hibi-higashitani-yoshida-tsuchiya}. If the polytope $\mathscr{P}$ satisfies the property that all of the coefficients of $\ehr(\mathscr{P},t)$ are nonnegative, then we say that $\mathscr{P}$ is \emph{Ehrhart positive}. 
For a detailed reading about Ehrhart positivity, we refer to surveys by Liu \cite{liu} and Ferroni--Higashitani \cite{ferroni-higashitani}. 

One of the most challenging endeavors in Ehrhart theory is that of finding classes of polytopes that are Ehrhart positive. As is suggested by equation~\eqref{eq:order-ehrhart}, one may ask similar questions about the sign patterns of order polynomials of posets, noting that if $\Omega(P;t)$ has nonnegative coefficients, then $\mathscr{O}(P)$ is Ehrhart positive, but not conversely. 

The Ehrhart positivity of order polytopes has attracted increasing attention during the past decade. As was noted by Stanley in \cite[Exercise~3.164]{ec1}, there exist posets whose order polynomials have negative coefficients. More strongly, there are order polytopes that fail to be Ehrhart positive. In fact, the results by Liu and Tsuchiya \cite{liu-tsuchiya} and Liu, Xin, Zhang \cite{liu-xin-zhang} imply that if $|P|< 14$ then $\mathscr{O}(P)$ is Ehrhart positive, but for every $n\geq 14$ there exists a poset $P$ on $n$ elements such that $\mathscr{O}(P)$ is not Ehrhart positive. The positivity of the coefficients of $\Omega(P;t)$ is even more rare, in the sense that one can find posets on $5$ elements (for example, an antichain of size $4$ covered by one element) attaining a negative coefficient. 

\subsection{Cell posets of skew shapes}
The central result of this article establishes the nonnegativity of the order polynomials of the cell posets of arbitrary skew shapes. In other words, we prove the following statement.

\newtheorem*{maintheorem}{Main Theorem}

\begin{thm}[Main Theorem~\ref{thm:positivity_skew}]\label{thm:main-positivity_skew}
    Let $\la/\mu$ be a skew shape of size $n$. The coefficients of the order polynomial $\Omega(P_{\la/\mu};t)$ are nonnegative.
\end{thm}

We now explain the notation in the preceding theorem. Consider any skew shape $\lambda/\mu$ represented as a skew Young diagram. The poset $P_{\la/\mu}$ is the poset whose elements are the boxes $X(\la/\mu) = \{ (i,j): i \in [\ell], j \in [\mu_i+1,\la_i]\}$ and covering relations given by $(i,j) \succeq (i+1,j)$, $(i,j) \succeq (i,j+1)$, see Figure~\ref{fig:example-poset}.

\begin{figure}[ht]
    \includegraphics[scale=0.85]{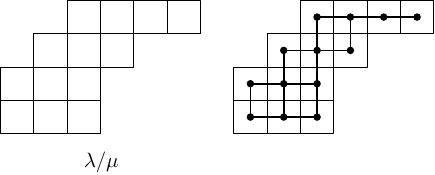}
 \quad  
 \raisebox{20pt}{
 \begin{tikzpicture}[scale=0.4,auto=center,every node/.style={circle,scale=0.7, fill=black, inner sep=2.7pt}] 
	\tikzstyle{edges} = [thick];
    \node (A1) at (6,3) {};
    \node (A2) at (4,3) {};
    \node (A3) at (2,3) {};

    \node (B1) at (7,2) {};
    \node (B2) at (5,2) {};
    \node (B3) at (3,2) {};
    \node (B4) at (1,2) {};
    
    \node (C1) at (8,1) {};
    \node (C2) at (6,1) {};
    \node (C3) at (4,1) {};
    \node (C4) at (2,1) {};

    \node (D1) at (9,0) {};
    \node (D2) at (3,0) {};

    \draw (A1) -- (B1);
    \draw (B1) -- (C1);
    \draw (C1) -- (D1);
    \draw (A2) -- (B2);
    \draw (B2) -- (C2);
    \draw (A3) -- (B3);
    \draw (B3) -- (C3);
    \draw (B4) -- (C4);
    \draw (C4) -- (D2);

    \draw (B1) -- (C2);
    \draw (A1) -- (B2);
    \draw (B2) -- (C3);
    \draw (C3) -- (D2);
    \draw (A2) -- (B3);
    \draw (B3) -- (C4);
    \draw (A3) -- (B4);
\end{tikzpicture}}
\quad 
\includegraphics[scale=0.85]{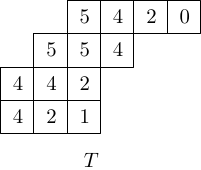}
    \caption{The skew shape $\lambda/\mu=6533/21$, its cell poset $P_{\lambda/\mu}$, and a plane partition with entries in $\{0,\ldots,5\}$.}
    \label{fig:example-poset}
\end{figure}
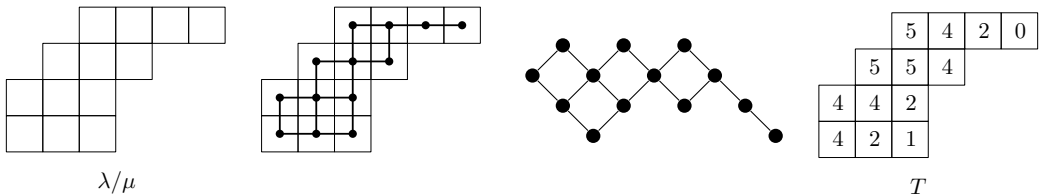

\subsection{Skew plane partitions and Kreweras determinants}

The order polynomial of a skew shape $\lambda/\mu$ enumerates plane partitions of shape $\lambda/\mu$ with bounded parts. A famous determinantal formula due to Kreweras~\cite{Kreweras_1965} for skew shapes and to MacMahon for straight shapes \cite{MacMahon} implies that
    \[ \Omega(P_{\lambda/\mu},t) =  \det\left[ \binom{t-1 + \lambda_i-\mu_j }{\lambda_i-\mu_j-i+j}
\right]_{i,j=1}^{\ell}\]
where $\ell$ denotes the length of $\lambda/\mu$. 

The content of Theorem~\ref{thm:main-positivity_skew} can be recast so to say that Kreweras determinants expand, in the standard basis $\{t^i\}_{i \geq 0}$, as polynomials with nonnegative coefficients. In the case of straight shapes, the nonnegativity of this expansion was proved by Fomin and Kirillov \cite[Theorem~2.1]{FK} in the context of Schubert calculus.

Moreover, one may reformulate Theorem~\ref{thm:main-positivity_skew} as a positivity phenomenon for the multichain enumerator on intervals of a Young's lattice. That is, our result asserts that the so-called \emph{zeta polynomial} of an arbitrary interval in Young's lattice has nonnegative coefficients (cf. \cite[Exercise~3.149]{ec1}).

\subsection{Fence posets and zig-zag posets} 

When the skew shape $\lambda/\mu$ is a ribbon, i.e., when its Young diagram does not contain a $2\times 2$ square, the cell poset $P_{\lambda/\mu}$ is commonly referred to as a \emph{fence poset}; see Figure~\ref{fig:fence poset} for an example. 

\begin{figure}[ht]
    \begin{tikzpicture}[scale=0.5,auto=center,every node/.style={circle,scale=0.7, fill=black, inner sep=2.7pt}] 
	\tikzstyle{edges} = [thick];
    \node (A1) at (3,0) {};

    \node (B1) at (2,1) {};
    \node (B2) at (4,1) {};
    \node (B3) at (6,1) {};
    
    \node (C1) at (1,2) {};
    \node (C3) at (5,2) {};
    \node (C4) at (7,2) {};

    \node (D1) at (0,3) {};
    \node (D2) at (8,3) {};

    \draw (D1) -- (C1);
    \draw (C1) -- (B1);
    \draw (B1) -- (A1);
    \draw (A1) -- (B2);
    \draw (B2) -- (C3);
    \draw (C3) -- (B3);
    \draw (B3) -- (C4);
    \draw (C4) -- (D2);
    
\end{tikzpicture}
\caption{Example of a fence poset.}
\label{fig:fence poset}
\end{figure}
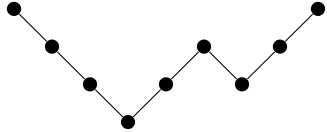

Fence posets have received considerable attention in the last few years, due to their relevance in the study of quiver representations, cluster algebras, and valuations in matroid theory. A relevant problem that garnered much interest was the unimodality conjecture due to Morier-Genoud and Ovsienko \cite{morier-genoud-ovsienko} for the rank polynomials of the lattices of ideals of fence posets. That conjecture has been settled by Kantarc{\i}\ O{\u g}uz and Ravichandran \cite{oguz-ravichandran}. For order polynomials, we obtain the following corollary of Theorem~\ref{thm:main-positivity_skew}.

\begin{corollary}[Corollary~\ref{coro:fences-order-positive}]\label{coro:fences-order-positive-main}
    Let $P$ be a fence poset. Then $\Omega(P;t)$ has nonnegative coefficients.
\end{corollary}

The above result was one of the main motivations to prove Theorem~\ref{thm:main-positivity_skew}. The first author raised this question\footnote{See Problem 1.2 here \url{http://aimpl.org/ehrhartineq/1/}} in a workshop at the American Institute of Mathematics in May of 2022 and, since then, various independent attempts of proving this statement have been made by several groups of researchers. Two special cases deserve an explicit mention. First, the case of fences with only two `arms', i.e., when $\mu=\varnothing$ and $\lambda$ is a hook: despite appearances, this simple case is already a fairly difficult problem. The nonnegativity of the order polynomial in that special case is one of the main results by Ferroni in \cite{ferroni_hooks}.
Second, the case of fence posets whose Hasse diagram is a zig-zag (see Figure~\ref{fig:zig-zag} below): the order polynomials of these posets have attracted considerable attention, e.g., in \cite{kirillov,chen-zhang,coons-sullivant,petersen-zhuang,lundstrom-saud}, even though the nonnegativity of the coefficients remained open, while now follows as a special case.

\subsection{Shard polytopes and Matroids}
Corollary~\ref{coro:fences-order-positive-main}  has implications at the level of a new fascinating class of polytopes called \emph{shard polytopes}, introduced in recent work by Padrol, Pilaud, and Ritter~\cite{padrol-pilaud-ritter}. 

\begin{corollary}[Corollary~\ref{coro:shards}]
    Shard polytopes of type A are Ehrhart positive.
\end{corollary}

An influential conjecture due to De Loera, Haws, and K\"oppe \cite{deloera-haws-koppe}, disproved by the first author in \cite{ferroni}, postulated that matroid polytopes were Ehrhart positive. Shard polytopes of type A are a special class of matroids, whose polytopes in fact are integrally equivalent to the order polytope of a fence poset. More so, as we shall explain, shard polytopes are exactly the base polytopes of \emph{snake matroids}. In particular, we have the following consequence.

\begin{corollary}\label{coro:ehrhart-snake}
    Let $\M$ be a matroid. If $\M$ is a direct sum of loops, coloops, and snake matroids, then its base polytope is Ehrhart positive.
\end{corollary}

\newtheorem*{metatheorem}{Meta Theorem}

\subsection{Classes of posets with positive order polynomials}
In order to prove Theorem~\ref{thm:main-positivity_skew} we demonstrate a general strategy to establish the nonnegativity of coefficients of order polynomials. This strategy can be summarized via the following statement. 

\begin{metatheorem}[Theorem~\ref{thm:meta_positivity}]\label{meta-theorem}
     Let $\mathcal{F}$ be a family of posets closed under taking lower and upper ideals, such that for every $P \in \mathcal{F}$ we have $c_1(P) := [t^1] \,\Omega(P;t) \geq 0$. Then $|P|!\cdot \Omega(P;t)$ has nonnegative integer coefficients for every $P\in \mathcal{F}$.
\end{metatheorem}

Many classes of posets appearing in combinatorics display the stability property required by the preceding statement, i.e., they are closed under taking upper and lower ideals. Hence, the nonnegativity of the coefficients of the order polynomial can be reduced to a single coefficient: the linear one.  

It is tempting to apply a similar strategy on some classes like planar posets, trees/forests, or even binary trees; but via small examples they can be ruled out very quickly; see Section~\ref{ss:nonexamples}. Nonetheless, based on evidence collected by computer experiments and proofs for specific examples, we conjecture that this general procedure may lead to the positivity of the order polynomial of other classes of posets.

The classes appearing in the conjecture below are easily seen to be closed under the operations of taking upper and lower order ideals, so it would suffice to establish the nonnegativity of the linear term.

\begin{conjecture}\label{conj:triple-conjecture}
    The following classes of posets have order polynomials with nonnegative coefficients.
    \begin{enumerate}[(i)]
        \item Posets of width two (see Conjecture~\ref{conj:width 2}).
        \item Cell posets of cylindric skew shapes (see Conjecture~\ref{conj:cylindric}).
        \item Cell posets of shifted skew shapes (see Conjecture~\ref{conj:shifted}).
    \end{enumerate}
\end{conjecture}

A poset has width two if the size of a maximal antichain is equal to $2$. A detailed definition and description for the cell posets of cylindric and shifted skew shapes is provided in Section~\ref{sec:other}.

For width two posets, we have verified by an exhaustive computation that this is true whenever the length of the maximal chains in $P$ is at most $5$.

In order to provide additional support to the conjecture on cylindric skew shapes we also prove that it holds for any circular fence poset, i.e., a fence poset in which we add a comparability relation between the `first' and the `last' element, has a nonnegative order polynomial. Like the classical fences mentioned earlier, circular fences also play a relevant role in the study of cluster algebras.

\begin{theorem}[Theorem~\ref{thm:circular-fences-order-positive}]
    Let $P$ be a circular fence poset. Then $\Omega(P;t)$ has nonnegative coefficients.
\end{theorem}

The analog of Kreweras formula that we employ in this setting is another determinantal formula for cylindric skew shapes by Gessel and Krattenthaler \cite{GesselKrattenthaler} (see Theorem~\ref{thm:GesselKrattCylPP}).

\subsection*{Acknowledgments}

The authors thank the American Institute of Mathematics, where they first met, and the Institute for Advanced Study, where this project was carried out, for the support provided. The authors also acknowledge many fruitful conversations about fence posets and Ehrhart theory with the participants of the AIM Workshop on Ehrhart polynomials held in May 2022. We thank Dave Anderson and Tianyi Yu for insightful conversations about Schubert polynomials and connections to plane partitions. We also thank Per Alexandersson, Swee Hong Chan, Sam Hopkins, and Hugh Thomas for helpful comments and suggestions. We are grateful to Haimu Wang for calculations that led to Conjecture~\ref{conj:width 2}, and Leonardo Saud-Maia-Leite for useful discussions about the circular zig-zag poset. We also thank the anonymous referees for their helpful comments and suggestions.

\section{Background}

We will assume that the reader is acquainted with the fundamental notions about partially ordered sets. However, below we recapitulate some of the essential concepts in this topic, especially those relevant for the content of this paper. We refer to \cite{ec1} for the undefined terminology and further reading.

\subsection{Order polynomials and order polytopes}
Let $P=(X,\preceq)$ be a partially ordered set and let $|X|=n$. 
Let $\mathcal{O}(P;S)$ be the set of order preserving maps $f:X \to S$. 
The \emph{order polynomial} of $P$ is the unique polynomial $\Omega(P;t)\in \mathbb{Q}[t]$ such that
\[\Omega(P;t):=\# \{ f:X \to \{1,\ldots,t\}: f(x) \leq f(y) \text{ if } x \preceq y\}=\#\mathcal{O}(P;\{1,\ldots,t\}),\]
for each positive integer $t$. The fact that this counting function is indeed a polynomial is a classical result~\cite[Section~3.12]{ec1}; one way to see this directly is through the formula
\begin{align}\label{eq:omega_binom}
    \Omega(P;t) = \sum_{k=1}^n \sum_{ \emptyset=I_0 \subsetneq \cdots \subsetneq I_k=P } \binom{t}{k},
\end{align}
where the sum ranges over all chains of ideals of $P$. If $P=\emptyset$ then we set $\Omega(P;t)=1$.
Note that when $P$ is not the empty poset, we have that $\Omega(P;0)=0$.

The \emph{order polytope} of $P$, introduced by Stanley in \cite{stanley-polytopes}, is defined by 
    \begin{equation}\label{eq:order-polytope}
        \mathscr{O}(P) := \{ f \in [0,1]^P : f(u) \leq f(v) \text{ for all $u\preceq v$ in $P$}\}.
    \end{equation}
This is a subpolytope of the unit cube $[0,1]^P$ whose vertices are in bijection with the order ideals of the poset. In particular, it is a lattice polytope so we may consider its Ehrhart polynomial $\ehr(\mathscr{O}(P),t)$, as defined in the introduction. The identity in equation~\eqref{eq:order-ehrhart} implies that both the order polynomial and the Ehrhart polynomial have the same leading coefficient, which in turn is equal to the relative volume of the polytope $\mathscr{O}(P)$. In turn, a result of Stanley in \cite{stanley-polytopes} implies that this equals (up to the normalization by $|P|!$) the number of linear extensions of $P$. A similar interpretation for the second highest coefficient of $\Omega(P;t)$ can be provided, see~\cite[Exercise~3.163]{ec1}.

\begin{example}\label{ex:faulhaber}
    Consider the poset $P_n$ arising from an antichain of size $n$ covering a single minimum element. For example, Figure~\ref{fig:faulhaber} depicts this poset for $n=4$. The order polytope $\mathscr{O}(P)$ is a pyramid over an $n$-dimensional cube. The order polynomial can be computed for a positive integer $t$ as follows:
        \[ \Omega(P_n; t) = \sum_{j=1}^{t} j^n,\]
    see, e.g., \cite[Example~2.8]{ferroni-higashitani}. The polynomial $\Omega(P_n;t)$ is known in the literature as the $n$-th \emph{Faulhaber's polynomial}.
    When $n = 4$, this is
        \[ \Omega(P_4;t) = \tfrac{1}{5} t^{5} + \tfrac{1}{2} t^{4} + \tfrac{1}{3} t^{3} - \tfrac{1}{30} t\]
    which has a negative linear term. The smallest $n$ such that $\ehr(\mathscr{O}(P_n),t) = \Omega(P_n;t+1)$ has negative coefficients is $n=20$, see \cite{liu-tsuchiya}.
    
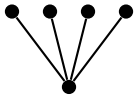
\begin{figure}[ht]
    \centering
	\begin{tikzpicture}  
	[scale=0.5,auto=center,every node/.style={circle,scale=0.7, fill=black, inner sep=2.7pt}] 
	\tikzstyle{edges} = [thick];
	
	\node[] (a1) at (0,0) {};  
	\node[] (a2) at (-1.5,2) {};  
	\node[] (a3) at (-0.5,2) {};  
        \node[] (a4) at (0.5,2) {};  
	\node[] (a5) at (1.5,2) {};

	\draw[edges] (a1) -- (a2);  
	\draw[edges] (a1) -- (a3);  
	\draw[edges] (a1) -- (a4);  
	\draw[edges] (a1) -- (a5); 
	\end{tikzpicture}\caption{One minimum element, covered by $4$ elements.}\label{fig:faulhaber}
\end{figure}
\end{example}

\subsection{Plane partitions and Kreweras determinants}

A \emph{plane partition} $T$ is a filling of the shape $\la/\mu$ with nonnegative integers that is weakly decreasing along rows and weakly decreasing along columns. See Figure~\ref{fig:example-poset}. It is immediate to verify that every plane partition corresponds an order preserving map $T:X \to \mathbb{Z}_{\geq 0}$.
We set $\PP_{\la/\mu}(t) := \#\mathcal{O}(P_{\lambda/\mu};\{0,\ldots,t\})$, so 
\begin{equation} \label{eq:omega of lam/mu is PP}
\Omega(P_{\la/\mu};t) = \PP_{\la/\mu}(t-1).
\end{equation}
Note that this is also the \emph{zeta polynomial} of the interval $[\mu,\lambda]$ in Young's lattice, see \cite[Ex.~3.149]{ec1}.

In \cite{Kreweras_1965}, Kreweras, gave the following determinantal identity for $\PP_{\la/\mu}(t)$. The case of straight shapes is due to MacMahon \cite{MacMahon}.

\begin{prop}[Kreweras \cite{Kreweras_1965}]\label{Kreweras_form}
Let $\PP_{\lambda/\mu}(t)$ be the number of plane partitions of shape $\lambda/\mu$ with entries less than or equal to $t$. We have that 
\begin{equation} \label{eq:krew det}
\PP_{\lambda/\mu}(t)  = \det\left[ \binom{\lambda_i-\mu_j +t}{\lambda_i-\mu_j-i+j}
\right]_{i,j=1}^{\ell}
\end{equation}
\end{prop}

\section{Order polynomials}

The goal in this section is to establish a general framework to prove the nonnegativity of order polynomials. As we will see, for certain classes of posets, this task may be reduced to just proving the positivity of \emph{one} coefficient. 

The first ingredient is the following elementary ``coproduct formula'' for order polynomials. This is well known to the experts and can be found, e.g., in \cite[Ex.~1.2.4(d)]{kung-yan} or \cite[Theorem~2.6]{shareshian2003newapproachorderpolynomials}. We include a short self-contained proof for completeness.

\begin{prop}\label{prop:omega_x_y}
    For variables $x,y$ we have the following polynomial identity
    $$\Omega(P;x+y) = \sum_{I \subset P} \Omega(I;x)\Omega(P \setminus I;y),$$
    where the sum goes over all ideals $I$ of $P$ (including $\emptyset$ and $P$). 
\end{prop}
\begin{proof}
    As both sides are polynomials in $x$ and $y$ it is enough to prove the identity when $x,y \in \mathbb{Z}_{\geq 0}$. It is done via the following bijection $\phi: \O(P;[x+y]) \to \bigcup_{I} \O(I;[x]) \times \O(P \setminus I;[y])$:
    for $f \in \O(P;[x+y])$ let $I= \{ a \in X: f(a) \leq x\}$, it is easy to see that $I$ is an order ideal. Then set $f|_I \in\O(I;[x])$ with $ f_I(a)=f(a)$, and $f|_{P\setminus I}(a) = f(a)-x \in[y]$, so $f|_{P\setminus I} \in \O(P\setminus I)$. Then set $\phi(f) = (f|_I , f|_{P \setminus I})$ with $I$ determined above. 
\end{proof}

\begin{remark}
    The zeta polynomial of a poset $P$ is a polynomial $Z(P;t)\in \mathbb{Q}[t]$ that enumerates multichains of elements in $P$ of size $t$. Defining the polynomial $Z([u,v];t)$ as the zeta polynomial of the interval $[u,v]\subseteq P$, it is easy to see that \[Z([u,w]; x+y) = \sum_{u\preceq v\preceq w} Z([u,v];x)Z([v,w];y).\] Now, it is a standard fact in poset theory that the order polynomial of a poset $P$ equals the zeta polynomial of its lattice of order ideals $J(P)$ (e.g., see \cite[Section 3.12]{ec1}). One can use this to derive another proof of the last proposition.
\end{remark}

\begin{prop}\label{prop:omega_coeffs}
    For any non-empty poset $P$, let $\Omega(P;t) = \sum_{k=1}^n c_k(P) t^k$ expand with coefficients $c_k(P)$. Then for every $k\geq 2$ we have that 
    $$c_k(P) = \frac{1}{2^k-2} \left( \sum_{I \subsetneq P, I\neq \emptyset} \sum_{i=1}^{k-1} c_i(I) c_{k-i}(P \setminus I) \right),$$
    where the sum goes over all nonempty ideals $I\neq P$.
\end{prop}
\begin{proof}
    Apply Proposition~\ref{prop:omega_x_y} with $x=y=t$, we have that
    $$\Omega(P;2t) = \sum_{I \subset P} \Omega(I;t)\Omega(P \setminus I;t).$$
    Subtracting the terms with $I=\emptyset$ and $I=P$ we have
    $$\Omega(P;2t) - 2\Omega(P;t) = \sum_{I \subsetneq P; I \neq \emptyset} \Omega(I;t)\Omega(P\setminus I;t).$$
Finally, expand both sides as polynomials in $t$ and compare coefficients at $t^k$, noting that
\[\Omega(P;2t)-2\Omega(P;t) = \sum_{k=1}^n c_k(P)2^kt^k - 2c_k(P)t^k=\sum_{k=2}^n (2^k-2)c_k(P) t^k.\qedhere\]
\end{proof}

Now let $\mathcal{F}$ be a family of posets closed under taking lower and upper ideals, i.e., if $P \in \mathcal{F}$ and $I \subset P$ is an order ideal then $I, P \setminus I \in \mathcal{F}$. Assume $\emptyset \in \mathcal{F}$ and necessarily $\{1\} \in \mathcal{F}$ (poset of 1 element). 

\begin{thm}\label{thm:meta_positivity}
    Let $\mathcal{F}$ be a family of posets closed under taking lower and upper ideals, such that for every $P \in \mathcal{F}$ we have $c_1(P) := [t^1] \, \Omega(P;t) \geq 0$. Then $c_k(P) \geq 0$ for every $k \geq 1$. If $n=|P|$, then $n!\Omega(P;t)$ has nonnegative integer coefficients.
\end{thm}
\begin{proof}
    The proof follows by strong induction on $n$, the size of the posets in the family. 
    We have $\Omega(\emptyset;t)=1$ and $\Omega(\{1\};t) =t$, both in $\mathbb{Z}_{\geq 0}[t]$.
    Suppose now that the coefficients of the order polynomial of every $P \in \mathcal{F}$ with $|P|\leq n-1$ are nonnegative. Let $P\in \mathcal{F}$ be a poset with $n$ elements.
    By assumption we have that $c_0(P)=0$ and $c_1(P) \geq 0$. For $k\geq 2$ we invoke the formula from Proposition~\ref{prop:omega_coeffs}, noting that $c_i(I) \geq 0$ and $c_{k-i}(P \setminus I) \geq 0$ since $I, P\setminus I \in \mathcal{F}$ and have at most $n-1$ elements each. Thus $c_k(P) \geq 0$ and this finishes the induction.

    Finally, observe that $n! \cdot \Omega(P;t) \in \mathbb{Z}[t]$ for every poset $P$. This can be easily seen from formula~\eqref{eq:omega_binom} since  $n! \cdot \binom{t}{k} = n(n-1)\cdots(k+1) t(t-1)\cdots(t-k+1) \in \mathbb{Z}[t]$. 
\end{proof}

\section{Positivity of skew plane partitions}\label{sec:pp}

Let $\mathcal{F}$ be the family of all posets $P_{\la/\mu}$ for all skew shapes $\la/\mu$. This family is closed under taking lower or upper ideals, since if $I \subset P_{\la/\mu}$ is an order ideal then $I=P_{\la /\nu}$ for some $\nu$ and $P \setminus I = P_{\nu/\mu}$. More specifically, if $(i,\nu_i) \in I$ is the last element in $I$ from row $i$ of $[\la/\mu]$ then $(i,j) \in I$ for all $j\leq \nu_i$, and $\nu_{i+1} \leq \nu_i$, or else $(i+1,\nu_{i+1}) \succeq (i,\nu_{i+1}) \not \in I$. 

We can then apply Theorem~\ref{thm:meta_positivity} if we show that $c_1(P_{\la/\mu}) \geq 0$ for every skew shape. 

\begin{prop}\label{prop:skew_coeff_1}
    Let $\la/\mu$ be a skew shape of size $n$ and length $\ell$ with order polynomial $\Omega(P_{\la/\mu};t) = \sum_{k=1}^n c_k(P)\,t^k$. Then if $\la/\mu$ is not connected, i.e., $\la_i=\mu_i$ or $\la_i \leq \mu_{i-1}$ for some $i$ we have $c_1(P_{\la/\mu})=0$. If $\la/\mu$ is connected then
    \[c_1(P_{\la/\mu}) = \frac{(\la_1-\mu_\ell-1)!  (\ell-1)!}{(\la_1-\mu_\ell-1+\ell)!}.\]
\end{prop}
\begin{proof}
    If $\ell(\la)=1$ then 
    $$\Omega(P_{\la/\mu};t) = \binom{t+\la_1-\mu_1-1}{\la_1-\mu_1} = \frac{t(t+1)\cdots (t+\la_1-\mu_1-1)}{(\la_1-\mu_1)!}$$
    and so $c_1(P_{\la/\mu}) = \frac{1}{\la_1-\mu_1}$.

    If $\la/\mu$ is not connected then it consists of two or more skew shapes $\theta^1, \theta^2,\ldots$ and $P_{\la/\mu}=P_{\theta^1}+P_{\theta^2}+\cdots$, where $P+Q$ denotes the direct sum of posets $P$ and $Q$. Then $\Omega(P_{\la/\mu};t) = \prod_i \Omega(P_{\theta^i};t)$ is divisible by $t^2$ since each of the terms is divisible by $t$ and so $c_1=0$. 
    
    For $\ell:=\ell(\la)>1$ we use Kreweras's formula from Proposition~\ref{Kreweras_form}:
    $$\Omega(P_{\la/\mu};t) = \det \left[ \binom{ \la_i-\mu_j+t-1}{\la_i-\mu_j-i+j} \right]_{i,j=1}^{\ell(\la)}.$$
    Denote the entries in this matrix by $A_{i,j}(t)$.
    In that formula, if $\la_i -\mu_j -i + j <0$ then the term is considered 0.
    For every $j \geq i$ we then have
    $$A_{i,j}(t)=\binom{\la_i - \mu_j +t-1}{\la_i-\mu_j-i+j} = \frac{ (t-1+\la_i-\mu_j) \cdots (t+i-j)}{(\la_i-\mu_j-i+j)!}.$$
    Since $\la_i >\mu_i \geq \mu_j$ then there is a term $t$ in the above product and the binomial is divisible by $t$, so $A_{i,j}(t)=tB_{i,j}(t)$ for $i \leq j$.    
    
     We now expand
     \[\det\left[ A_{i,j}(t) \right]_{i,j=1}^\ell = \sum_{w \in \mathfrak{S}_\ell} \operatorname{sign}(w) \prod_i A_{i,w(i)}(t).\]
     The products are all divisible by $t$ since $w(i) \geq i$ for some $i$ always. In order to get a linear term $t$ then exactly one of $A_{i,w(i)}(t)$ should be divisible by $t$, i.e., there is exactly one $i$ for which $i\leq w(i)$. We must then have that $w(i)=\ell$ since $\ell \geq w(i)$ for all $i$, and then $w(j) <j$ for all $j \neq i$. Since $w(1) \not<1$ we must then have $i=1$, so $w(1)=\ell$. Since the remaining entries in the permutation matrix of $w$ should be below $i=j$, we must have $w(i)=i-1$ for all $i=2,\ldots,\ell$. Then the only term in the determinant expansion which gives a linear term is
     $$A_{1,\ell}(t) \prod_{i=2}^\ell A_{i,i-1}(t) = \binom{t-1+\la_1-\mu_\ell}{\la_1-\mu_\ell-1+\ell} \prod_{i=2}^\ell \binom{t-1+\la_i -\mu_{i-1}}{\la_i - \mu_{i-1}-1}$$
We have $ \binom{t-1+\la_i -\mu_{i-1}}{\la_i - \mu_{i-1}-1} = \frac{ (t+1)\cdots(t+\la_i-\mu_{i-1}-1)}{(\la_i-\mu_{i-1}-1)!} = 1 + O(t)$ and 
\begin{align*}
\binom{t-1+\la_1-\mu_\ell}{\la_1-\mu_\ell-1+\ell} &= \frac{ (t+\la_1-\mu_\ell-1) \cdots (t-\ell+1)}{(\la_1-\mu_\ell-1+\ell)!}\\
&= t \frac{(\la_1-\mu_\ell-1)! (-1)^{\ell-1} (\ell-1)!}{(\la_1-\mu_\ell-1+\ell)!} + O(t^2)
\end{align*}
Since $\operatorname{sign}(w)=(-1)^{\ell-1}$ we obtain the formula.
 \end{proof}

We can now derive Theorem~\ref{thm:main-positivity_skew} by combining Proposition~\ref{prop:skew_coeff_1} and the Theorem~\ref{thm:meta_positivity}.

\begin{thm}\label{thm:positivity_skew}
    Let $\la/\mu$ be a skew shape of size $n$. Then the coefficients of its order polynomial $\Omega(P_{\la/\mu};t)$ are all nonnegative and $|\la/\mu|! \cdot \Omega(P_{\la/\mu};t) \in \mathbb{Z}_{\geq 0}[t]$.
\end{thm}

\begin{rem}
    Order polynomials have a generalization using the theory of $(P,\omega)$-partitions where $\omega$ is a labelling of $P$ \cite[Section~3.15]{ec1}. Let $\Omega(P,\omega;t)$ be the {\em $(P,\omega)$-order polynomial}, which counts the number of $(P,\omega)$ partitions $\sigma:P\to \{1,\ldots,t\}$ when $t \in \mathbb{Z}_{\geq 1}$. Proposition~\ref{prop:omega_x_y} holds when $\omega$ is a natural labelling. However, it is not true that in general $\Omega(P_{\lambda/\mu},\omega;t)$ has nonnegative coefficients. For instance, if $\omega$ is a {\em Schur labelling} \cite[Section~7.19]{EC2}, then $\Omega(P_{\lambda/\mu},\omega;t)$ counts the number of \emph{semistandard Young tableaux} of shape $\lambda/\mu$ with entries $0,\ldots,t-1$. From the hook-content formula \cite[Cor.~7.21.4]{EC2} or from a direct calculation one can check that 
    \[
    \Omega(P_{11},\omega;t) = \binom{t}{2} = \frac{t^2}{2}-\frac12.
    \]
    In particular, there is no extension of the last theorem to general labellings.
\end{rem}

\section{Fence posets, shard polytopes, and snake matroids}

\subsection{Fence posets} \label{sec:fence posets}

When the skew shape $\lambda/\mu$ has a Young diagram that does not contain a $2\times 2$ square, we say it is a \emph{ribbon} or a \emph{border strip}. The posets that arise as cell posets of border strips are called \emph{fence posets}. They appear naturally in the study of quiver representations, cluster algebras and, as we will explain below, matroid theory. See \cite{mcconville-sagan-smyth} for a more detailed discussion on some of these connections.

Fence posets have attracted considerable attention from combinatorialists in the recent years. One of the main open problems regarding fences was a conjecture by Morier-Genoud and Ovsienko \cite{morier-genoud-ovsienko} asserting the unimodality of the rank polynomials of all fence posets. This conjecture was confirmed by Kantarc{\i}\ O{\u g}uz and Ravichandran in 
\cite{oguz-ravichandran}. 

Specializing Theorem~\ref{thm:main-positivity_skew} to the case in which $\lambda/\mu$ is a ribbon shape leads to the following result.

\begin{corollary}\label{coro:fences-order-positive}
    Fence posets have order polynomials with nonnegative coefficients.
\end{corollary}

Among all fence posets, the most pervasive is the zig-zag poset, which corresponds to a fence poset with covering relations of the form $x_1\succ x_2 \prec x_3 \succ x_4 \prec \cdots \succ x_{2n}$ or $x_1\succ x_2 \prec x_3 \succ x_4 \prec \cdots \prec x_{2n+1}$. In this paper we will denote by $Z_n$ the zig-zag poset on $n$ elements. For a depiction of the Hasse diagram of $Z_9$, see Figure~\ref{fig:zig-zag}.

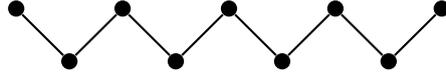
\begin{figure}[ht]
    \centering
	\begin{tikzpicture}  
	[scale=0.7,auto=center,every node/.style={circle,scale=0.8, fill=black, inner sep=2.7pt}] 
	\tikzstyle{edges} = [thick];

    \node[] (a0) at (-1,1) {}; 
	\node[] (a1) at (0,0) {};  
	\node[] (a2) at (1,1) {};  
	\node[] (a3) at (2,0) {};  
        \node[] (a4) at (3,1) {};  
	\node[] (a5) at (4,0) {};  
	\node[] (a6) at (5,1) {};  
	\node[] (a7) at (6,0) {};  
	\node[] (a8) at (7,1) {};  

     \draw[edges] (a0) -- (a1); 
    \draw[edges] (a1) -- (a2);
	\draw[edges] (a2) -- (a3);
	\draw[edges] (a3) -- (a4);
	\draw[edges] (a4) -- (a5);
	\draw[edges] (a5) -- (a6);
	\draw[edges] (a6) -- (a7);
	\draw[edges] (a7) -- (a8);

	\end{tikzpicture}\caption{The Hasse diagram of the zig-zag poset $Z_9$.}\label{fig:zig-zag}
\end{figure}

The order polynomial and the order polytope of $Z_n$ have been studied in several articles, including \cite{stanley-polytopes,kirillov,chen-zhang,coons-sullivant,petersen-zhuang}. Despite its relevance, not many results regarding $\Omega(Z_n;t)$ have been obtained. Moreover, it has been a long-standing folklore open problem to prove that $\mathscr{O}(Z_n;t)$ is Ehrhart positive. Evidently, now we can obtain a proof as an easy special case of Theorem~\ref{thm:positivity_skew} since $Z_n = P_{\zeta_n}$, where $\zeta_n$ is a zig-zag ribbon of size $n$:
\begin{equation} \label{eq:zigzag skew shape}
\zeta_n  = \begin{cases} 
 (k+1,k,\ldots,2)/(k-1,k-2,\ldots,1) &\text{ if } n=2k\\
(k+1,k,\ldots,1) /(k-1,k-2,\ldots,1) &\text{ if } n=2k+1.
\end{cases}
\end{equation}

\begin{corollary}
    For every $n\geq 2$, the zig-zag poset $Z_n$ has an order polynomial with nonnegative coefficients. In particular, $\mathscr{O}(Z_n)$ is Ehrhart positive.
\end{corollary}Even though we have confirmed the nonnegativity of the order polynomials of zig-zag posets, we do not have a combinatorial interpretation for the coefficients of $n!\cdot \Omega(Z_n;t)$. By Proposition~\ref{Kreweras_form} we have that 
\begin{equation} \label{eq:krewzigzag}
    \Omega(Z_n;t)= \det\left[ \binom{t+1-i+j}{2j-2i+2}
\right]_{i,j=1}^{k},
\end{equation}
where $k=\lfloor n/2 \rfloor$. Using this determinant formula or by a direct computation one can see that 
\begin{equation} \label{eq:ex order poly zig-zags}
     n!\cdot \Omega(Z_n;t) = \begin{cases}
        t & n = 1,\\
        t^2 + t & n = 2,\\
        2t^3 + 3t^2 + t & n = 3,\\
        5t^4 + 10t^3 + 7t^2 + 2t & n = 4,\\
        16t^5 + 40t^4 + 40t^3 + 20t^2 + 4t & n = 5,\\
        61t^6 +  183 t^5 + 235 t^4 +  165 t^3 +  64t^2 +12t & n=6,\\
        \text{etc.} & 
    \end{cases}
    \end{equation}
Since the order polynomial of every poset evaluated at $t=1$ gives the value $1$, it is no surprise that the coefficients on the right-hand-side sum to a factorial. It is well known that the leading terms are the so-called \emph{Euler numbers} $E_n$ \cite[\href{https://oeis.org/A000111}{A000111}]{OEIS}, the second leading term is $E_n\cdot n/2$ and, by Proposition~\ref{prop:skew_coeff_1}, the term of degree $1$ is always $\lfloor \frac{n-1}{2}\rfloor ! \cdot \lceil \frac{n-1}{2}\rceil!$ \cite[\href{https://oeis.org/A010551}{A0110551}]{OEIS}. All of these quantities enumerate certain permutations on $n$ elements. We leave the following open problem:

\begin{problem}
    Find a statistic on permutations of size $n$ whose generating function gives the polynomials $n!\cdot \Omega(Z_n;t)$.
\end{problem}

\subsection{Matroids}

In this section we will address Ehrhart polynomials of matroids. Recall that a matroid $\M$ is a pair $(E,\mathcal{B})$ where $E$ is a finite set and $\mathcal{B}\subseteq 2^E$ is a nonempty collection of sets called bases that satisfies the following exchange axiom:
    \[ \text{for all $B_1,B_2\in\mathcal{B}$, $i\in B_1\smallsetminus B_2$ there exists $j\in B_2\smallsetminus B_1$ such that $(B_1\smallsetminus i)\cup j\in \mathcal{B}$}.\]
Henceforth we will adopt the terminology and notation of \cite{oxley}, and we refer to it for any undefined terminology regarding matroids. 

Let $\M$ be a matroid on $E=[n]$. The \emph{matroid base polytope} of $\M$ is the polytope $\mathscr{P}(\M)\subseteq \mathbb{R}^n$ defined by
    \[ \mathscr{P}(\M) := \text{convex hull} ( e_B : B\in \mathscr{B}),\]
where $e_B := \sum_{i\in B} e_i$ stands for the indicator vector of the basis of $\M$. 

Ehrhart polynomials of matroid base polytopes have been intensively studied during the last two decades. The main motivation was the Ehrhart positivity conjecture posed by De Loera, Haws, and K\"oppe in \cite[Conjecture~2]{deloera-haws-koppe}, and a stronger version by Castillo and Liu in \cite{castillo-liu} for all generalized permutohedra. Even though these conjectures were disproved in \cite{ferroni}, in many interesting special cases Ehrhart positivity indeed holds true. Some instances appearing in the literature are uniform matroids \cite{ferroni1}, matroids of rank $2$ \cite{ferroni-jochemko-schroter}, panhandle matroids \cite{mcginnis,deligeorgaki-mcginnis-vindas}, Catalan matroids \cite{chen-li-yao}, generalized permutohedra that are Minkowski sums of standard simplices \cite{postnikov}, generalized permutohedra of dimension at most~$6$ \cite{castillo-liu}, and various special matroids of small size. More so, the linear term of the Ehrhart polynomials\footnote{We note that the linear term of the Ehrhart polynomial of an order polytope equals $\sum_{i} i c_i(P)$, where $c_i(P)$ denotes the degree $i$ coefficient in $\Omega(P;t)$.} of matroids and generalized permutohedra has been proved nonnegative by Jochemko and Ravichandran \cite{jochemko-ravichandran} and, using different tools, by Castillo and Liu \cite{castillo-liu2}. Further related articles include \cite{knauer-martinez-ramirez,benedetti-knauer-valencia,fan-li,hanely,dania}.

The main focus in this section is to explain which matroids can be proved to be Ehrhart positive by applying Theorem~\ref{thm:main-positivity_skew}. To this end we recapitulate the essentials on the class of \emph{snake matroids}, assuming that the reader is familiar with the concept of lattice path matroid (see \cite{bonin-demier-noy}).

\begin{definition}
    A matroid $\M$ is said to be a \emph{snake matroid} if $\M$ satisfies the following three properties:
    \begin{enumerate}[(i)]
        \item $\M$ has at least two elements and is connected.
        \item $\M$ is isomorphic to a lattice path matroid.
        \item $\M$ does not have a $\mathsf{U}_{2,4}$ minor, i.e., $\M$ is binary.
    \end{enumerate}
\end{definition}

We refer to \cite{knauer-martinez-ramirez}, \cite{benedetti-knauer-valencia}, \cite{dania}, and \cite[Appendix~A]{ferroni-schroter} for a detailed discussion on snakes. When presented as lattice path matroids, snakes have diagrams that look like border strips. The following result is essentially a reformulation of a handy result by Benedetti, Knauer, and Valencia-Porras\footnote{In their work, the authors use the term ``affinely equivalent'', but all the maps appearing in their proof are piecewise linear transformations that preserve the lattice.} appearing in \cite[Theorem~3.3]{benedetti-knauer-valencia}.

\begin{theorem}
    Let $P$ be a finite poset. The following conditions are equivalent.
    \begin{enumerate}[\normalfont(i)]
        \item $\mathscr{O}(P)$ is integrally equivalent to a matroid polytope.
        \item $\mathscr{O}(P)$ is integrally equivalent to a matroid polytope coming from a direct sum of loops, coloops, and snakes.
        \item $P$ is a disjoint union of fences and single element posets.
    \end{enumerate}
\end{theorem}

Notice that if two polytopes are integrally equivalent, then their Ehrhart polynomials must coincide. In particular, the preceding theorem says that the only Ehrhart polynomials that Theorem~\ref{thm:main-positivity_skew} captures within the realm of matroids are precisely those of the direct sums of snakes, loops and coloops. In particular, we obtain the following result.

\begin{corollary}\label{coro:matroids}
    Let $\M$ be a matroid that is a direct sum of snakes, loops, and coloops. Then, $\mathscr{P}(\M)$ is Ehrhart positive.
\end{corollary}

In \cite[Appendix~A]{ferroni-schroter} Ferroni and Schr\"oter proved that direct sums of snake matroids, loops and coloops, span the valuative group of matroids introduced by Eur, Huh, and Larson \cite{eur-huh-larson}. In other words, the Ehrhart polynomial of every matroid can be written as an integer combination of Ehrhart polynomials of matroids that are direct sums of snakes, loops and coloops. As a consequence of Corollary~\ref{coro:matroids}, we observe the following property of Ehrhart polynomials of matroid base polytopes: they are nonnegative on a spanning set of the valuative group, yet they fail to be nonnegative in general due to the counterexamples found in \cite{ferroni}. To the best of the authors' knowledge, this is the first valuative invariant of matroids arising `in nature'\footnote{A different, much more ad-hoc, example is described in \cite[Theorem~6.1]{ferroni-fink}, where in fact positivity holds for all realizable matroids.} for which this phenomenon shows.

We close this section with one intriguing open problem. All snake matroids are series-parallel matroids. However, the converse is false, as can be seen by taking the graphic matroid induced by a $3$-cycle in which every edge is doubled in parallel---that matroid is series-parallel, but not isomorphic to a lattice path matroid. We do not know how to extend the Ehrhart positivity of snakes to all series-parallel matroids.

\begin{problem}
    Show that all series-parallel matroids are Ehrhart positive.
\end{problem}

We note that this is a special case of a conjecture by Ferroni, Jochemko, and Schr\"oter \cite[Conjecture~6.2]{ferroni-jochemko-schroter} on the Ehrhart positivity of positroids. 

\begin{remark}
    From our main result, it follows that Corollary~\ref{coro:matroids} admits the following upgrade: the Ehrhart polynomial of a matroid that is a direct sum of snakes and loops/coloops is nonnegative when seen as a polynomial in $t-1$. This property is not true for positroids. For example, the hypersimplex $\Delta_{2,9}$ is a positroid but its Ehrhart polynomial at $t-1$ has a negative cubic coefficient.
\end{remark}

\begin{remark}
    After the completion of this paper, Chavez, Dorpalen-Barry, Ferroni, Liu, Rinc\'on, and Vindas-Meléndez \cite{square} found a surprising application of the calculations carried out in Section~\ref{sec:pp}. These authors were able to prove that the $\beta$-invariant of any matroid without loops and coloops can be read off of the Ehrhart polynomial of the base polytope. It is utterly opaque to see the property proved in \cite{square} from just the definition of the Ehrhart polynomial and the $\beta$-invariant. Currently, the only known proof relies on Proposition~\ref{prop:skew_coeff_1} in an essential way.
\end{remark}

\subsection{Shard polytopes} The family of {\em shard polytopes} (of type A) were introduced in \cite{padrol-pilaud-ritter} by Padrol--Pilaud--Ritter in the context of lattice quotients of the weak order of the symmetric group.

Shard polytopes can be defined as follows. An {\em arc} is a tuple $\alpha=(a,b,A,B)$ for integers $a$ and $b$ with $1\leq a <b\leq n+1$, and $A$, $B$ a partition of the set $[a,b-1]:=\{a,a+1,\ldots,b-1\}$. Given an arc $\alpha$, a {\em $\alpha$-fall} ({\em $\alpha$-rise} resp.) as a position $j \in [a,b-1]$ such that $j\in \{a\}\cup A$ and $j+1\in B\cup \{b\}$ ($j\in \{a\} \cup B$ and $j+1\in A \cup \{b\}$ resp.). Given an arc $\alpha$, the {\em shard polytope} $\SP(\alpha)$ has the following H-description,
\[
\SP(\alpha) \,=\, \left\{ (x_1,\ldots,x_n) \in \mathbb{H}  \quad{\Bigg |}\quad \begin{array}{l} x_i=0 \text{ for } i\in [n]\setminus [a,b],\\ x_{a'}\geq 0 \text{ for } a'\in A, x_{b'}\leq 0 \text{ for } b' \in B,\\ \sum_{i\leq f} x_i \leq 1 \text{ for an $\alpha$-fall $f$},\\ \sum_{i\leq r} x_i \geq 0 \text{ for an $\alpha$-rise $r$} \end{array}  \right\},
\]
where $\mathbb{H}$ is the hyperplane $\{(x_1,\ldots,x_n)  \in \mathbb{R}^n \mid \sum_i x_i =0\}$.

The following theorem is known among the specialists. However, to the best of our knowledge, this is the first time it is stated in this form.

\begin{proposition}
    Shard polytopes of type A are exactly the base polytopes of snake matroids.
\end{proposition}

\begin{proof}
    A characterization of snake matroids as graphic matroids was found by Knauer, Ram\'irez-Alfons\'in and Mart\'inez-Sandoval in \cite[Theorem~2.2]{knauer-martinez-ramirez0}. These are a special class of series-parallel graphs that they call `multi-fans'. In the original paper by Padrol, Pilaud, and Ritter \cite[Proposition~72]{padrol-pilaud-ritter}, characterizes shard polytopes using the exact same class of series-parallel graphs. 
\end{proof}

The authors in \cite{LACIM_g_vector} also show that shard polytopes of type A are (integrally equivalent to) order polytopes of fences via another family of polytopes called {\em Harder--Narasimhan polytopes} \cite{BKT} coming from finite dimensional quiver representations.

\begin{corollary}\label{coro:shards}
    Shard polytopes of type A are Ehrhart positive.
\end{corollary}

\section{Schubert shenanigans}\label{sec:schubert}

In this section, whose title is inspired by \cite{stanley2017some}, we go over how a connection to Schubert polynomials and plane partitions of Fomin and Kirillov \cite{FK} and an identity of Macdonald also give the positivity of $\PP_{\lambda}(t-1)$. 

\subsection{Permutations and Schubert polynomials}

We denote permutations $w$ of $\{1,2,\ldots,n\}$ in \emph{one-line notation} $w=w_1w_2\cdots w_n$. Given a positive integer $m$, we denote by $1^m\times w$ the permutation $12\cdots m (m+w_1)(m+w_2)\cdots (m+w_n)$. The symmetric group of permutations of size $n$ is generated by the simple transpositions $s_1,\ldots,s_{n-1}$ where $s_i=(i,i+1)$ is the transposition that exchanges $i$ and $i+1$. A \emph{reduced decomposition} of a permutation $w$ is a factorization of $w$ into a product of generators $w=s_{i_1}s_{i_2}\cdots s_{i_{\ell}}$ with $\ell=\ell(w)$ is minimal. This quantity is called the \emph{length} of $w$. The {\em reduced words} of a permutation $w$ of size $n$ are tuples $(i_1,i_2,\ldots,i_{\ell})$ in $[n-1]^{\ell}$ such that $s_{i_1}s_{i_2}\cdots s_{i_{\ell}}$ is a reduced decomposition of $w$. We denote by $\RW(w)$ the set of reduced words of the permutation $w$. Recall the relations of the simple transpositions: $s_i^2=\iota$, the commutation $s_is_j=s_js_i$ if $|i-j|>1$, and the {\em braid relation} $s_is_{i+1}s_i=s_{i+1}s_is_{i+1}$. A well known fact \cite[Prop. 2.1.6]{ManivelBook} is that all the reduced words can be obtained from one another by a sequence of commutations and braid moves.

To a permutation $w$ of size $n$ we associate its \emph{(Rothe) diagram} $D(w)=\{(i,w_j) \mid i<j, w_i>w_j\}$. The size of this set is the length $\ell(w)$ of the permutation. From this diagram, one can obtain a {\em lexicographically minimal reduced word} $r'$ as follows: for each row $i=1,\ldots,n$ label the cells on row $i$ of the diagram from left to right $i, i+1,\ldots$. Read off the word from the labelled diagram from top to bottom, reading from right to left.

\begin{example} \label{ex: Rothe diagrams and lex min reduced words}
The Rothe diagram of the permutation $w=4231$ is $$D(w)=\{(1,1),(1,2),(1,3),(2,1),(3,1)\}.$$ The lexicographically minimal reduced word of $w$ is $r'=(3,2,1,2,3)$. The Rothe diagram of the permutation $w=461532$ is $$D(w)=\{(1,1),(1,2),(1,3),(2,1),(2,2),(2,3),(2,5),(4,2),(4,3),(5,2)\}.$$ The lexicographically minimal reduced word of $w$ is $r'=(3,2,1,5,4,3,2,5,4,6)$. See Figure~\ref{fig:examples diagrams perms} (left) and (center).   
\end{example}

We define two families of permutations that are of interest. They can also be described in terms of {\em pattern avoidance}. A permutation $w$ is \emph{dominant} (or $132$-avoiding) if the diagram $D(w)$ is a left justified (flushed  left) Young diagram of a partition $\lambda\subseteq \delta_n:=(n-1,n-2,\ldots,1)$; i.e., $D(w)=\{(1,1),\ldots,(1,\lambda_1),\ldots,(r,1),\ldots,(r,\lambda_r)\}$ for a partition $\lambda=(\lambda_1,\lambda_2,\ldots,\lambda_r)\subseteq \delta_n$. The size $n$ and $\lambda\subseteq \delta_n$ uniquely determine the dominant permutation that we denote by $w_{\lambda}(n)$. If $n$ is not specified, we say that $w_{\lambda}$ is a dominant permutation of shape $\lambda$. See Figure~\ref{fig:examples diagrams perms} (left). A permutation is \emph{vexillary} (or $2143$-avoiding) if $D(w)$ is, up to permuting rows and columns, the Young diagram of a partition $\mu(w)$. Given a vexillary permutation $w$, let $\lambda(w)$ be the smallest partition containing the diagram $D(w)$ (this partition can also be defined in terms of the {\em essential set} of $D(w)$). Since $\mu(w)\subset \lambda(w)$, we associate to a vexillary permutation $w$ the skew shape $\lambda(w)/\mu(w)$. See Figure~\ref{fig:examples diagrams perms} (center).  Note that for a dominant $w_{\mu}$ permutation and a positive integer $m$, the permutation $v=1^m\times w_\mu$ is vexillary with associated skew shape $\lambda(v)/\mu(v)$ where $\lambda(v) = ((m+\mu_1)^m, m+\mu_1, m+\mu_2,\ldots,m+\mu_{r})$ and $\mu(v)=\mu$. See Figure~\ref{fig:examples diagrams perms} (right).

\begin{remark} \label{rem:Edelman--Greene}
The number of reduced words of a vexillary permutation $w$ equals the number of 
{\em standard Young tableaux} of shape $\mu(w)$ \cite{StanleyCoxeter} which encode the linear extensions of the cell poset $P_{\mu(w)}$. There is a variant of the well-known Robinson--Schensted--Knuth (RSK) correspondence called the {\em Edelman--Greene correspondence} \cite{EG} between the reduced words of $w$ and these standard tableaux.
\end{remark}


Schubert polynomials were defined by Lascoux and Sch\"utzenberger \cite{LSSchub} to study the intersections of Schubert varieties. We denote by $\mathfrak{S}_w({\bf x})=\mathfrak{S}_w(x_1,\ldots,x_{n-1})$ the \emph{(single) Schubert polynomial} of $w$. See \cite{ManivelBook,MacdonaldSchubBook,billey2025introductioncohomologyflagvariety} for the definition  of this polynomial.

\subsection{A connection with work of Fomin--Kirillov}

The following result of Fomin and Kirillov \cite{FK}  gives an identity for $\PP_{\lambda}(t)$ as a weighted sum over reduced words of the dominant permutation. The result is obtained in part by using Macdonald's identity \cite[Eq. (6.11)]{MacdonaldSchubBook} for the evaluation $\mathfrak{S}_w(1,\ldots,1)$ of a Schubert polynomial of $w$ in terms of reduced words. 

\begin{thm}[{Fomin--Kirillov \cite[Thm.~2.1]{FK}}] \label{thm:FK}
Let $\lambda$ be a partition of $n$ then
\begin{equation}  \label{eq:Macdonald identity PP} \tag{FKM} 
\PP_{\lambda}(t) = \frac{1}{n!} \sum_{r} (t+r_1)\cdots (t+r_n),
\end{equation}
where the sum is over reduced words $r_1\cdots r_n$ of a dominant permutation associated to $\lambda$. In particular both $n!\cdot \PP_{\lambda}(t)$ and $ n!\cdot \PP_{\lambda}(t-1)$ are in $\mathbb{Z}_{\geq 0}[t]$.
\end{thm}

We will refer to \eqref{eq:Macdonald identity PP} as the Fomin--Kirillov--Macdonald identity.


\begin{ex}
For the hook $\lambda=(3,1,1)$, the smallest associated dominant permutation is $4231$ with reduced words $(3,2,1,2,3),
(3,1,2,3,1),
 (1,3,2,3,1),
 (3,1,2,1,3),
 (1,3,2,1,3),(1,2,3,2,1)
 $.

\begin{figure}[ht]
 \begin{center}
     \includegraphics[scale=0.7]{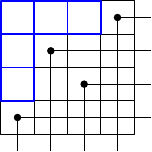}
     \qquad 
     \includegraphics[scale=0.7]{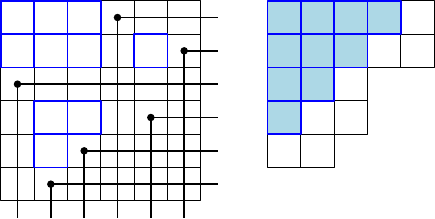}
     \qquad
     \includegraphics[scale=0.7]{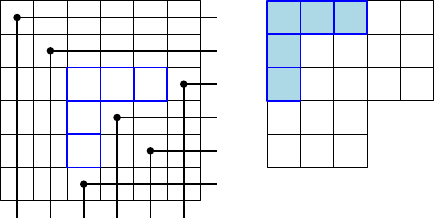}
 \end{center}
 \caption{The diagram of a dominant permutation $w_{311}(4)=4231$ (left), and the diagram and associated skew shape of the vexillary permutations $w=461532$ (center) and of $1^2\times w_{311}(4)=126453$ (right). }
 \label{fig:examples diagrams perms}
 \end{figure}
 
 Then 
 \begin{align*}
 \PP_{\lambda}(t) &= \frac{1}{5!}\left(2(t+1)^2(t+2)^2(t+3) + 4(t+1)^2(t+2)(t+3)^2\right)\\
 &= \frac{1}{120}(6t^5 + 60t^4 + 230t^3 + 420t^2 + 364t + 120). 
 \end{align*}    
\end{ex}




We use the Fomin--Kirillov--Macdonald identity to obtain a more explicit formula for the hook shape $\lambda=(a+1,1^b)$. The positivity of $\PP_{\lambda}(t-1)$ in this case was deduced by Ferroni in \cite[Thm. 1.6]{ferroni_hooks}.

\begin{cor}
For $\lambda=(a+1,1^b)$ with $a\leq b$, we have that 
\begin{align}
\PP_{\lambda}(t) &= \frac{1}{(a+b+1)!} \sum_{i=1}^{a+1} \binom{a}{i-1}\binom{b}{i-1} (t+i)\prod_{j \in [a+1] \setminus i} (t+j) \prod_{j \in [b+1] \setminus i} (t+j), \label{eq: Mac hook 1}\\
&= \frac{(t+a+1)_{a+1} (t+b+1)_{b+1}}{(a+b+1)!}\sum_{i=1}^{a+1} \frac{1}{t+i} \binom{a}{i-1}\binom{b}{i-1}. \label{eq: Mac hook 2}
\end{align}
\end{cor}

\begin{proof}
By applying the Fomin--Kirillov--Macdonald's identity \eqref{eq:Macdonald identity PP} to the hook $\lambda=(a+1,1^b)$ we obtain that 
\begin{equation} \label{eq: Macdonald hook proof}
    \PP_{\lambda}(t) = \frac{1}{(a+b+1)!} \sum_{r} (t+r_1)\cdots (t+r_{a+b+1}),
\end{equation}
where the sum is over the reduced words of a dominant permutation $w_{\lambda}$, say $w_{\lambda}(b+2)$.
From the Edelman--Greene correspondence (see Remark~\ref{rem:Edelman--Greene}), the number of reduced words of $w_{\lambda}$ equals the number of linear extensions of cell poset $P_{\lambda}$ which is $\binom{a+b}{b}$. Next, we determine the reduced words of this permutation. From the diagram of the permutation we can read the canonical reduced word: \[r':=(a+1,a,\ldots,2,{\bf 1},2,3,\ldots,b+1).\] 
The other reduced words are obtained from $r'$ by repeated application of braid moves and commutations. The braid moves will happen only at positions $a,a+1,a+2$ and will be of the form $(i+1,i,i+1)\to (i,i+1,i)$ for $i=1,\ldots,a$. If the letter in the $(a+1)$th position is $i$, for $i=1,\ldots,a+1$, then the letters on the left of the $(a+1)$th position are  $[a+1]\setminus \{i\}$ and by using commutations, the subword is a shuffle of $a+1,a,\ldots,i+1$ and $1,2,\ldots,i-1$. Similarly, the letters on the right of the $(a+1)$th position are $[b+1]\setminus \{i\}$ and by  using commutations, the subword is a shuffle of $b+1,b,\ldots,i+1$ and $1,2,\ldots,i-1$. Thus, there are $\binom{a}{i-1}\cdot \binom{b}{i-1}$ reduced words where the letter on the $(a+1)$th position is $i$ (and indeed, $\sum_{i=1}^{a+1} \binom{a}{i-1}\binom{b}{i-1} = \binom{a+b}{b}$, the total number of reduced words). The underlying multiset of these reduced words is $[a+1]\setminus \{i\} \cup \{i\} \cup [b+1]\setminus \{i\}$. Thus, their contribution to the sum on the RHS of \eqref{eq: Macdonald hook proof} is 
\[
\binom{a}{i-1}\binom{b}{i-1} (t+i)\prod_{j \in [a+1] \setminus i} (t+j) \prod_{j \in [b+1] \setminus i} (t+j).
\]
By summing this contribution for all the reduced words with letter $i$ in the $(a+1)$th position for $i=1,\ldots,a+1$ we obtain  \eqref{eq: Mac hook 1}. The relation \eqref{eq: Mac hook 2} is obtained from it by factoring out the falling factorials $(t+a+1)_{a+1}\cdot (t+b+1)_{b+1}$.
\end{proof}

We can use Fomin--Kirillov--Macdonald's identity \eqref{eq:Macdonald identity PP} to obtain a concrete expansion for $\PP_{\lambda}(m)$. Let $e_k:=\sum_{i_1<i_2<\cdots< i_k} x_{i_1}x_{i_2}\cdots x_{i_k}$ denote the $i$th elementary symmetric function.

\begin{cor} \label{cor:formulas order poly straight shape from Schub}
For a partition $\lambda$ of size $n$, we have that $\PP_{\lambda}(t-1) = \sum_{i=1}^n c_it^i$ where
\[
c_i = \frac{1}{n!}\sum_{{\bf r} \in \RW(w_{\lambda})} e_{n-i}(r_1-1,\ldots,r_n-1),
\]
where the sum is over all the reduced words ${\bf r}$ of $w_{\lambda}$. 

\end{cor}

\begin{proof}
The result follows from the Fomin--Kirillov--Macdonald identity \eqref{eq:Macdonald identity PP} and the identity (see \cite[Section 7.6]{EC2}),
\[
(t+r_1-1)\cdots (t+r_n-1) = \sum_{i=0}^n t^{i}e_{n-i}(r_1-1,\ldots,r_n-1).\qedhere
\]
\end{proof}

\begin{rem}
Combining the formula for $c_1(P_{\lambda})$ in Proposition~\ref{prop:skew_coeff_1} and Corollary~\ref{cor:formulas order poly straight shape from Schub}, we obtain the following curious identity: for a partition $\lambda=(\lambda_1,\ldots,\lambda_{\ell})$ of $n$ we have that 
\begin{equation}
\frac{1}{n!}\sum_{{\bf r} \in \RW_1(w_{\lambda})} \prod_{j\neq k} (r_j-1) = 
\frac{(\lambda_1-1)!(\ell-1)!}{(\lambda_1-1+\ell)!},
\end{equation}
where $\RW_1(w_\lambda)$ are the reduced words of $w_\lambda$ with exactly one index $k$ such that $i_k=1$.
\end{rem}




Next, we inquire about the possibility of having a Fomin--Kirillov--Macdonald type formula for the order polynomial of skew plane partitions. 

\begin{defn}
We say that a polynomial $f \in \mathbb{Z}_{\geq 0}[t]$ \emph{has a Fomin--Kirillov--Macdonald-type formula} for a set $S$, denoted as $F \in M[S]$ if $f$ can be written as a sum of terms $\prod_{i=1}^{\deg(f)} (t+a_i)$ where $a_i \in S$.
\end{defn}

As follows from the above discussion, the order polynomial for straight shapes $\lambda$ is in $M[S]$ where $S=\{0,\ldots,\lambda_1-1\}$. We will now explore the possibility of such formula for skew shapes.

Let $\zeta_n$ be the zig-zag skew shape with $n$ boxes from equation~\eqref{eq:zigzag skew shape} so that $Z_n = P_{\zeta_n}$. Let $S_n=\{0,\ldots,n\}$
We have that the order polynomials normalized by $|\la/\mu|!$ in \eqref{eq:ex order poly zig-zags} can be expressed as follows: 

\[
n!\Omega(Z_n;t) = \begin{cases}
    t(t+1) & n = 2,\\
    t^2(t+1) + t(t+1)^2  & n = 3,\\
    t(t+1)^2(t+2) + 2t^3(t+1) + 2t^2(t+1)^2 & n = 4,\\
    4t(t+1)^4 + 4t^2(t+1)^3 + 4t^4(t+1) + 4t^3(t+1)^2 & n = 5,\\
    12t(t + 1)^5 + t^2(t + 1)^2(t + 2)^2 + 33t^3 (t + 1)^3 + 12 t^5(t+1) + 
 3 t^4 (t + 1)^2 & n = 6,\\
    \text{etc.} & 
\end{cases}
\]

In other words, for $n \in \{2,3\}$ this polynomial lies in $M[S_1]$, and for $n\in \{4,5,6\}$ it lies in $M[S_2]$. However, it is easy to see (also experimentally) that $4!\cdot \Omega(Z_4;t) \not \in M[S_1]$. 
Notice that these expressions are not unique for larger values of $n$. 

\begin{conj}\label{conj:expansion}
    The order polynomials $|\lambda/\mu|! \cdot \Omega(\lambda/\mu;t)$ for $\la_1\geq \ell(\la)$ are in $M[S_{\lambda_1}]$ for all such $\lambda,\mu$, i.e., there is a multiset of compositions  of $|\lambda/\mu|$ denoted $R(\lambda/\mu)$, such that
    $$|\lambda/\mu|! \cdot \Omega(P_{\lambda/\mu};t) = \sum_{a \in R(\lambda/\mu) } \prod_{i=0}^{\lambda_1} (t+i)^{a_i}$$
\end{conj}

Ultimately, we leave open how to extend this Schubert approach to show the positivity of $\PP_{\lambda}(t-1)$. This highlights the novelty of the approach in the previous section to derive such a positivity result.

\section{Other classes of posets with nonnegative order polynomials?}\label{sec:other}

Here we consider several classes of partially ordered sets, which (experimentally) seem to exhibit order polynomials with nonngeative coefficients. These classes of posets are also naturally closed under taking upper and lower order ideals, so Theorem~\ref{thm:meta_positivity} would apply to them and so all we would need to prove is that $c_1(P) \geq 0$. 

\subsection{Width two posets}

A poset has width two if the longest antichain has two elements. Such posets are an interesting testing ground for conectures regarding linear extensions of posets like the {\em $\frac13$ -- $\frac23$ conjecture} \cite{LinialW2} and the {\em cross-product conjecture} \cite{ChanPakPanovaW2}. A width two poset can be viewed as two chains $\alpha_1 \succ \alpha_2 \succ\cdots\succ\alpha_m$, $\beta_n\succ\beta_{n-1}\succ\cdots\succ\beta_1$, and relations $\alpha_i \succ \beta_j$ or $\beta_j\succ\alpha_i$. Such width two posets $P$ are in correspondence with Young diagrams of a skew shape $\lambda/\mu$ inside the $m\times n$ rectangle \cite[\S 8]{ChanPakPanovaW2}, where it is phrased in the language of NE lattice paths. Namely, the cells $(i,j)$ of the Young diagram consist of incomparable pairs $(\alpha_i,\beta_j)$.  By abuse of notation, we denote the width two poset corresponding to a skew shape $\lambda/\mu\subset [m]\times [n]$ by $R_{\lambda/\mu}$. See Figure~\ref{fig:width 2 poset} (left). In \cite[Lemma 8.1]{ChanPakPanovaW2}, Chan--Pak--Panova found a correspondence between {\em linear extensions} of the width two poset $R_{\lambda/\mu}$ and NE lattice paths within the shape, which naturally corresponds to plane partitions of shape $\lambda/\mu$ with entries $0,1$. The latter are of course counted by $\PP_{\lambda/\mu}(1)$. We refer to a recent paper by Alexandersson and Jal \cite{alexandersson-jal} for a connection with rook placements on Ferrers boards and lattice path matroids.

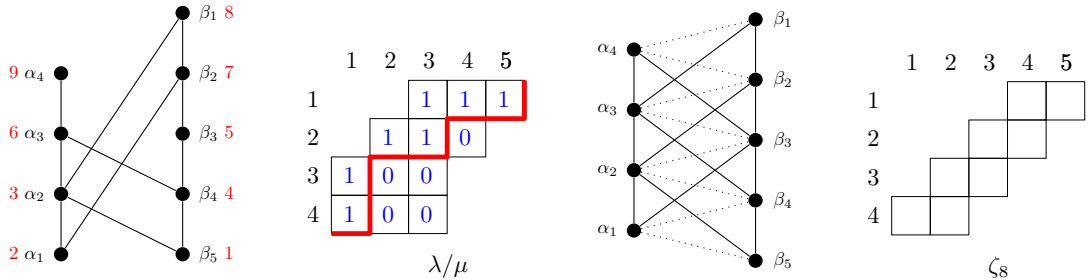
\begin{figure}[ht]
\centering
\begin{tikzpicture}[scale=0.4,auto=center,every node/.style={circle,scale=0.7, fill=black, inner sep=2.7pt}] 
	\tikzstyle{edges} = [thick];
   \node (A1) at (0,0) [label=left:{\textcolor{red}{2} $\alpha_1$ }] {};
    \node (A2) at (0,2) [label=left:{\textcolor{red}{3} $\alpha_2$ }] {};
    \node (A3) at (0,4) [label=left:{\textcolor{red}{6} $\alpha_3$ }] {};
    \node (A4) at (0,6) [label=left:{\textcolor{red}{9} $\alpha_4$ }] {};

    \node (B1) at (4,8) [label=right:{$\beta_1$ \textcolor{red}{8}}] {};
    \node (B2) at (4,6) [label=right:{$\beta_2$ \textcolor{red}{7}}] {};
    \node (B3) at (4,4) [label=right:{$\beta_3$ \textcolor{red}{5}}] {};
    \node (B4) at (4,2) [label=right:{$\beta_4$ \textcolor{red}{4}}] {};
    \node (B5) at (4,0) [label=right:{$\beta_5$ \textcolor{red}{1}}] {};
    
    \draw (A1) -- (A2);
    \draw (A2) -- (A3);
    \draw (A3) -- (A4);
    
    \draw (B1) -- (B2);
    \draw (B2) -- (B3);
    \draw (B3) -- (B4);
    \draw (B4) -- (B5);

    \draw (A1) -- (B2);
    \draw (A2) -- (B1);
    \draw (A3) -- (B4);
    \draw (A2) -- (B5);
\end{tikzpicture}
\qquad  
\includegraphics[scale=0.9]{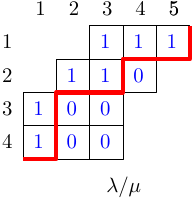}
\qquad   
\begin{tikzpicture}[scale=0.4,auto=center,every node/.style={circle,scale=0.7, fill=black, inner sep=2.7pt}] 
	\tikzstyle{edges} = [thick];
   \node (A1) at (0,1) [label=left:{$\alpha_1$ }] {};
    \node (A2) at (0,3) [label=left:{ $\alpha_2$ }] {};
    \node (A3) at (0,5) [label=left:{ $\alpha_3$ }] {};
    \node (A4) at (0,7) [label=left:{ $\alpha_4$ }] {};

    \node (B1) at (4,8) [label=right:{$\beta_1$} ] {};
    \node (B2) at (4,6) [label=right:{$\beta_2$} ] {};
    \node (B3) at (4,4) [label=right:{$\beta_3$ }] {};
    \node (B4) at (4,2) [label=right:{$\beta_4$ }] {};
    \node (B5) at (4,0) [label=right:{$\beta_5$ }] {};
    
    \draw (A1) -- (A2);
    \draw (A2) -- (A3);
    \draw (A3) -- (A4);
    
    \draw (B1) -- (B2);
    \draw (B2) -- (B3);
    \draw (B3) -- (B4);
    \draw (B4) -- (B5);

    \draw (A1) -- (B3);
    \draw (A2) -- (B2);
    \draw (A3) -- (B1);

    \draw (B5) -- (A2);
    \draw (B4) -- (A3);
    \draw (B3) -- (A4);

    \draw[dotted] (A1) -- (B5);
    \draw[dotted] (A1) -- (B4);
    \draw[dotted] (A2) -- (B4);
    \draw[dotted] (A2) -- (B3);
    \draw[dotted] (A3) -- (B3);
    \draw[dotted] (A3) -- (B2);
    \draw[dotted] (A4) -- (B2);
    \draw[dotted] (A4) -- (B1);
\end{tikzpicture}
\qquad 
\includegraphics[scale=0.9]{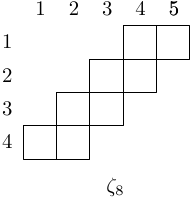}

\caption{A width two poset and its associated skew shape $\lambda/\mu=5433/21$ (left). The poset has  a linear extension in \textcolor{red}{red} and its associated plane partition with entries $0,1$ (\textcolor{blue}{blue}) using the correspondence from \cite[\S 8]{ChanPakPanovaW2}. The complement zigzag $\overline{Z}_9$ and its associated skew shape $\zeta_8=5432/321$ (right).}
\label{fig:width 2 poset}
\end{figure}

Based on calculations using {\tt SageMath} \cite{sagemath} for width two posets with chains of sizes $m,n \leq 5$, it appears that the order polynomial of a width two poset has  nonnegative coefficients.

\begin{conj}\label{conj:width 2}
The order polynomial of a width two poset $R_{\lambda/\mu}$ encoded by skew shape $\lambda/\mu\subset [m]\times [n]$ has nonnegative coefficients, i.e., $(m+n)! \cdot \Omega(R_{\lambda/\mu};t) \in \mathbb{Z}_{\geq 0}[t]$. 
\end{conj}

The class of width two posets is closed under taking lower and upper ideals and so Theorem~\ref{thm:meta_positivity} can be applied. Thus it is enough to prove that $c_1(R_{\lambda/\mu})$ is always nonnegative.

\begin{remark}
In addition to computational evidence, by Theorem~\ref{thm:main-positivity_skew}, the above conjecture holds for width two posets that are isomorphic to $P_{\lambda/\mu}$, i.e., for skew shapes $\lambda/\mu=(a+b+c,a+b)/(a)$, and hooks $\lambda/\mu=(a+1,1^b)$.
\end{remark}

\begin{remark}
An interesting width two poset is the {\em complement zigzag} $\overline{Z}_n$ from \cite[Example 10]{EulerFib} which has elements $x_1,\ldots,x_n$ and covering relations $x_i \preceq x_{i+2}$ for $i=1,\ldots,n-2$ and $x_i \preceq x_{i+3}$ for $i=1,\ldots,n-3$. See Figure~\ref{fig:width 2 poset} (right). One can see that $\overline{Z}_n \cong R_{\zeta_{n-1}}$ and that the leading term of $n!\cdot \Omega(\overline{Z}_n;t)$ is the $n$th Fibonacci number $F_n$. We have verified that the order polynomial of $\overline{Z}_n$ has nonnegative coefficients for $n\leq 200$.
\end{remark}

\subsection{Cylindric partitions}

A cylindric plane partition of skew diagram $\lambda/\mu/d$ is a plane partition $T$ that remains a plane partition if the last row, shifted by $d$ units to the right, is placed on top of $\pi$. See Figure~\ref{fig:ex cylindric plane partition}. These plane partitions were introduced by Gessel and Kratthenthaler in \cite{GesselKrattenthaler} and are related to affine Schubert calculus \cite{PostnikovCylindric}, {\em periodic Schur processes} \cite{BorodinCylindric}, {\em Rogers--Ramanujan identities} of integer partitions \cite{CorteelDousseAli}, and chromatic symmetric functions \cite{siegl2022cylindricptableaux31freeposets} among others. The sequence of nonnegative integers $\lambda=(\lambda_1,\ldots,\lambda_\ell)$, $\mu=(\mu_1,\ldots,\mu_\ell)$ need to satisfy the following condition
\begin{align} \label{eq: cylindric conditions lambda mu}
\lambda_1 &\geq \lambda_2-1 \geq \lambda_3-2 \geq \cdots \geq \lambda_\ell - (\ell-1) \geq \lambda_1 -d-\ell,\\
\mu_1 &\geq \mu_2-1 \geq \mu_3-2 \geq \cdots \geq \mu_{\ell}- (\ell-1) \geq \mu_1 -d-\ell. \notag
\end{align}

\begin{figure}[ht]
    \includegraphics[scale=0.9]{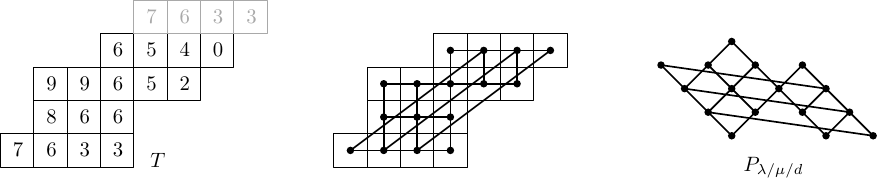}
    \caption{An illustration of a cylindric plane partition of shape $\lambda/\mu/d=7644/311/4$ and the Hasse diagram of the associated poset $P_{\lambda/\mu/d}$.}\label{fig:ex cylindric plane partition}
\end{figure}

For $t$ in $\mathbb{Z}_{\geq 0}$, let $\CPP_{\lambda/\mu/d}(t)$ be the number of cylindric plane partitions of shape $\lambda/\mu/d$ with entries $0,\ldots,t$. Let $P_{\lambda/\mu/d}$ be the underlying cell poset of a cylindric partition of shape $\lambda/\mu/d$. That is,  $P_{\la/\mu/d}$ is the poset whose elements are the boxes $X(\la/\mu) = \{ (i,j): i \in [\ell], j \in [\mu_i+1,\la_i]\}$ and covering relations given by $(i,j) \succeq (i+1,j)$, $(i,j) \succeq (i,j+1)$, as well as  $(1,\mu_1+d+k)\succeq (\ell, \mu_{\ell}+k)$ for $k=1,\ldots,\min(\lambda_1-d,\mu_{\ell})$.

Note that $\Omega(P_{\lambda/\mu/d};t)=\CPP_{\lambda/\mu/d}(t-1)$. In \cite{GesselKrattenthaler}, Gessel and Krattenthaler found the following formula for $\CPP_{\lambda/\mu/d}(t)$ as a sum of Kreweras type determinants.

\begin{theorem}[{Gessel--Krattenthaler \cite[Thm. 2]{GesselKrattenthaler}}] \label{thm:GesselKrattCylPP}
For a cylindric skew shape $\lambda/\mu/d$ for sequences of nonnegative integers $\lambda$ and $\mu$ satisfying \eqref{eq: cylindric conditions lambda mu} we have that 
\begin{equation}\label{eq:gk}
\Omega(P_{\lambda/\mu/d};t) = \sum_{(k_1,\ldots,k_{\ell})} \det\left[ \binom{t-1 + \lambda_i-\mu_j -dk_i }{\lambda_i-\mu_j-i+j-(\ell+d)k_i}
\right]_{i,j=1}^{\ell},
\end{equation}
where the sum is over tuples of integers $(k_1,\ldots,k_{\ell})$ such that $k_1+\cdots + k_{\ell}=0$.
\end{theorem}

\begin{remark}
The number of terms in the RHS of \eqref{eq:gk} is finite because, in order to have a nonvanishing determinant, in every row of the matrix we need to have at least one nonzero entry. Thus for $i=1,\ldots,\ell$ we have that $k_i \leq (\lambda_i-\mu_j -i+j)/(\ell+d)$ for some $j$ and subsequently that $k_i \geq -(\ell-1)(\lambda_i-\mu_j -i+j)/(\ell+d)$ for some $j$.    
\end{remark}

When the skew shape $\lambda/\mu$ is a ribbon or border strip and $d=\lambda_1-1$, the shape $\lambda/\mu/d$ is called a {\em closed ribbon}, and the poset $P_{\lambda/\mu/d}$ is called {\em circular fence poset}, see Figure~\ref{fig:closed zig-zag} for an example. These posets played an instrumental role in the proof by Kantarc{\i}\ O{\u g}uz and Ravichandran
\cite{oguz-ravichandran} mentioned in Section~\ref{sec:fence posets}. 

\begin{figure}[ht]
    \centering
    \begin{tikzpicture}  
	[scale=0.6,auto=center,every node/.style={circle,scale=0.8, fill=black, inner sep=2.7pt}] 
	\tikzstyle{edges} = [thick];
	
	\node[] (a0) at (-1,1) {};  
    \node[] (a1) at (0,0) {};  
	\node[] (a2) at (1,-1) {};  
	\node[] (a3) at (2,0) {};  
    \node[] (a4) at (3,1) {};  
	\node[] (a5) at (4,2) {};  
	\node[] (a6) at (5,1) {};  
	\node[] (a7) at (6,0) {};  

    \draw[edges] (a0) -- (a1);
    \draw[edges] (a1) -- (a2);
	\draw[edges] (a2) -- (a3);
	\draw[edges] (a3) -- (a4);
	\draw[edges] (a4) -- (a5);
	\draw[edges] (a5) -- (a6);
	\draw[edges] (a6) -- (a7);
	\draw[edges] (a0) -- (a7);
	\end{tikzpicture}
    \qquad
	\begin{tikzpicture}  
	[scale=0.6,auto=center,every node/.style={circle,scale=0.8, fill=black, inner sep=2.7pt}] 
	\tikzstyle{edges} = [thick];
	
	\node[] (a0) at (-1,1) {};  
    \node[] (a1) at (0,0) {};  
	\node[] (a2) at (1,1) {};  
	\node[] (a3) at (2,0) {};  
    \node[] (a4) at (3,1) {};  
	\node[] (a5) at (4,0) {};  
	\node[] (a6) at (5,1) {};  
	\node[] (a7) at (6,0) {};  

    \draw[edges] (a0) -- (a1);
    \draw[edges] (a1) -- (a2);
	\draw[edges] (a2) -- (a3);
	\draw[edges] (a3) -- (a4);
	\draw[edges] (a4) -- (a5);
	\draw[edges] (a5) -- (a6);
	\draw[edges] (a6) -- (a7);
	\draw[edges] (a0) -- (a7);
	\end{tikzpicture}
    \qquad \qquad \includegraphics{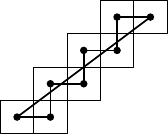}
    \caption{The Hasse diagram of a circular fence poset. The Hasse diagram of the circular zig-zag poset $\widehat{Z}_8$
    and the associated cylindric skew shape $5432/321/4$.
    }\label{fig:closed zig-zag}
\end{figure}
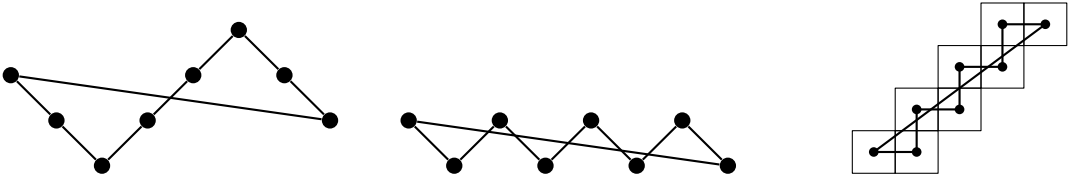

In the next result, we do a careful analysis of Theorem~\ref{thm:GesselKrattCylPP} to show that order polynomials of circular fence posets indeed have nonnegative coefficients. We start with a more compact expression for circular fences.

\begin{proposition} \label{prop:GK for border strips}
For a circular fence poset $P_{\lambda/\mu/\lambda_1-1}$, we have that
\begin{multline} \label{eq:sum_dets_ribbons}
\Omega(P_{\lambda/\mu/\lambda_1-1};t) \,=\, \Omega(P_{\lambda/\mu};t) + \\
+ \sum_{m=2}^n \det\left[ \binom{t-1 + \lambda_i-\mu_j - (\lambda_1-1)(\delta_{i,1}-\delta_{i,m})}{\lambda_i-\mu_j -i +j - (\ell + \lambda_1-1)(\delta_{i,1}-\delta_{i,m})}\right]_{i,j=1}^{\ell},   
\end{multline}
where $\delta_{a,b}$ is the Kronecker delta.
\end{proposition}

\begin{proof}
Since $\lambda/\mu/(\lambda_1-1)$ is a closed ribbon then  $\mu_{\ell}=0$. A term in the RHS of \eqref{eq:gk} has a  nonvanishing determinant if and only if for every $i=1,\ldots,\ell$, we have that $\lambda_i-\mu_j-i+j-(\ell+\lambda_1-1)k_i \geq 0$ for some $j$. Thus, either $k_i=0$ for all $i$ in which case we obtain the Kreweras determinant formula \eqref{eq:krew det} for a ribbon shape, or else $k_i>0$ for some $i$. In this case, we must have a $j$ such that
\[
\lambda_i - \mu_j -i +j \geq (\ell+\lambda_1-1)k_i.
\]
However, the conditions $\ell-1 \geq j-i$ and $\lambda_1 \geq \lambda_i-\mu_j$ force that $k_i=1$ and that $\ell-1=j-i$. This implies that $i=1$ and $j=\ell$. This in turn forces $k_m=-1$ for some $m=2,\ldots,\ell$. This gives the remaining summands on the RHS of \eqref{eq:sum_dets_ribbons} yielding the desired identity for $\Omega(P_{\lambda/\mu/\lambda_1-1};t)$. 
\end{proof}


\begin{theorem}\label{thm:circular-fences-order-positive}
    Circular fence posets  have order polynomials with nonnegative coefficients.
\end{theorem}

\begin{proof}
Let $P_{\lambda/\mu/\lambda_1-1}$ be a circular fence poset. We analyze the resulting determinant's linear term in Proposition~\ref{prop:GK for border strips} using the same approach as in Section~\ref{sec:pp}. Let $A(m)$ be the matrices appearing in the equation~\eqref{eq:sum_dets_ribbons}, so 
$$A(m)_{i,j} = \binom{t-1 + \lambda_i-\mu_j - (\lambda_1-1)(\delta_{i,1}-\delta_{i,m})}{\lambda_i-\mu_j -i +j - (\ell + \lambda_1-1)(\delta_{i,1}-\delta_{i,m})}$$ for $m\geq 2$. 

We have that $c_1(P_{\lambda/\mu})=
\frac{(\la_1-1)!  (\ell-1)!}{(\la_1-1+\ell)!}$ by Proposition~\ref{prop:skew_coeff_1}. We will also show that 
$$[t^1]\det A(m) = \frac{(\la_1-1)!  (\ell-1)!}{(\la_1-1+\ell)!}.$$ 

We recall some useful facts from the proof of Proposition~\ref{prop:GK for border strips}. First, we have that $\mu_{\ell}=0$ and that $\ell-1 \geq j-i$ and $\lambda_1\geq \lambda_i-\mu_j$. Thus, the only nonzero entry in the first row of the matrix $A(m)$ is $A(m)_{1,\ell}=\binom{t}{0}=1$.
Since $\lambda/\mu$ is a ribbon then we must have $\mu_i = \lambda_{i+1}-1$ for $i=1,\ldots,\ell(\lambda)-1$, so $\lambda_i-\mu_j-i+j = \lambda_i - i -(\lambda_{j+1}-j-1) \geq 0$ if and only if $j+1\geq i$ as $\lambda_i-i$ is a strictly decreasing sequence. By the same consideration, we have that for $i \neq 1,m$, the term $\binom{t-1+\lambda_i-\mu_j}{\lambda_i-\mu_j-i+j}$ for $i>j$ has a constant term and for $i\leq j$ is divisible by $t$. Therefore, the submatrix $A(m)|_{2:\ell,1:\ell-1}$ (rows $i=2,\ldots,\ell$ and columns $j=1,\ldots,\ell-1$) has entries equal to 0 below the diagonal on the rows where $i\neq m$. The  diagonal terms for $i \neq m$ are  $\binom{t+\lambda_i-\mu_j-1}{\lambda_i-\mu_j-1}=1 + O(t)$. On row $i=m$ the term on the diagonal is  
\[
A(m)_{m,m-1} = \binom{t+\lambda_m-\mu_{m-1}+\lambda_1-2}{\lambda_m-\mu_{m-1}+\ell+\lambda_1-2} = \binom{t+\lambda_1-1}{\lambda_1-1+\ell}= (-1)^{\ell-1}\frac{(\la_1-1)!  (\ell-1)!}{(\la_1-1+\ell)!}t + O(t^2),
\] 
where we used the fact that 
$\lambda_i-\mu_{i-1}=1$ for $i=2,\ldots,\ell$ since $\lambda/\mu$ is a ribbon. The off-diagonal terms below the diagonal are 
$$A(m)_{m,j}=\binom{ t-1+\lambda_m -\mu_j+\lambda_1-1}{\lambda_m-\mu_j-m+j+\ell+\lambda_1-1}=O(t)$$
since $\lambda_m-\mu_j+\lambda_1-2 <\lambda_m-\mu_j+j-m+\ell+\lambda_1-1$.
Computing the determinant of $A(m)|_{2:\ell,1:\ell-1}$ by expanding the minors along the row $i=m$, the only term which is not divisble by $t^2$ would come from the product of its diagonal entries, which is $A(m)_{m,m-1}$
Thus,

\[\det A(m) = (-1)^{\ell-1} A(m)_{m,m-1} +  O(t^2),\]
which has the desired linear term. 
Then 
\[c_1(P_{\lambda/\mu/\lambda_1-1}) = \ell \frac{(\la_1-1)!  (\ell-1)!}{(\la_1-1+\ell)!},\]
and applying Theorem~\ref{thm:meta_positivity} on the family of circular fence posets and fence posets we obtain the claim.
\end{proof}

\begin{example} \label{ex:cyclic zig-zag}
For an even positive integer $n$, let $\widehat{Z}_{n}$ be the circular zig-zag with $n$ elements, as depicted in Figure~\ref{fig:closed zig-zag}. 
For small values of $n$, the order polynomial takes the following values:
\begin{equation} \label{eq:ex order poly closed zig-zags}
     n!\cdot \Omega(\widehat{Z}_n;t) = \begin{cases}
        t^2 + t & n = 2,\\
         4t^4 + 8t^3 + 8t^2 + 4t & n = 4,\\
        48t^6 + 144t^5 + 210t^4 + 180t^3 + 102t^2 + 36t & n = 6,\\
        \text{etc.} & 
    \end{cases}
    \end{equation}
The coefficients on the right-hand-side sum to $(2n)!$ and the leading term counts the number of linear extensions of $\widehat{Z}_{2n}$ which is given by $n\cdot E_{2n-1}$ \cite[{\href{https://oeis.org/A024255}{A024255}}]{OEIS}. 
Now we can deduce that the order polynomial of the circular zig-zag has nonnegative coefficients as a special case of Theorem~\ref{thm:circular-fences-order-positive} since $\widehat{Z}_n$ is a circular fence poset. Various recurrences for $\Omega(\widehat{Z}_n;t)$ and a different proof of the positivity of the linear term of $\Omega(\widehat{Z}_n;t)$ not requiring the determinantal formula by Gessel--Krattenthaler appears in \cite{lundstrom-saud}.
\end{example}

Based on calculations using Theorem~\ref{thm:GesselKrattCylPP}, it appears that the order polynomial of  a cylindric skew shape has nonnegative coefficients.

\begin{conj}\label{conj:cylindric}
    The order polynomial of any cylindric  skew shapes has  nonnegative coefficients, i.e.
    $|\la/\mu|! \cdot \Omega(P_{\la/\mu/d};t) \in \mathbb{Z}_{\geq 0}[t].$
\end{conj}

\subsection{Shifted plane partitions}
Another well-known generalization of plane partitions of skew shape is plane partitions of shifted skew shape. A partition $\lambda$ with distinct parts $\lambda_1>\lambda_2>\cdots >\lambda_{\ell}$ is called {\em strict}. Given a strict partition $\lambda$, its {\em shifted Young diagram} is a Young diagram of $\lambda$ where we indent each row one position to the right of the row above it. Given two strict partitions $\lambda$ and $\mu$ such that the shifted Young diagram of $\mu$ is contained in the shifted Young diagram of $\lambda$, a {\em shifted skew shape} $\lambda/\mu$ is the set difference of shifted Young diagrams. A {\em shifted plane partition} of shifted skew shape $\lambda/\mu$ is a filling of the shifted Young diagram of $\lambda/\mu$ with nonnegative integers that is weakly decreasing along rows and columns. See Figure~\ref{fig:ex shifted plane partition}. Shifted plane partitions are related to symmetry classes of plane partitions \cite{StanleyDozen} and Schur $P$-functions and Schur $Q$-functions \cite{StembridgeProjRepr,HoffmanHumphreys,HAMEL1996328}. 

For $t$ in $\mathbb{Z}_{\geq 0}$, let $\ShPP_{\lambda/\mu}(t)$ be the number of shifted plane partitions of shifed skew shape $\lambda/\mu$ with entries $0,\ldots,t$. Let $Q_{\lambda/\mu}$ be the underlying cell poset of a shifted skew shape $\lambda/\mu$. Note that $\Omega(Q_{\lambda/\mu};t)=\ShPP_{\lambda/\mu}(t-1)$.

\begin{figure}[ht]
    \includegraphics[scale=0.9]{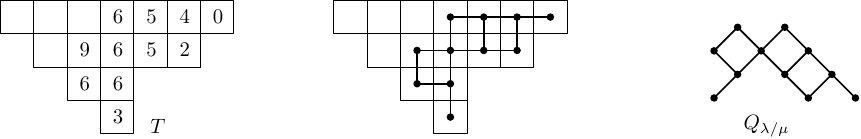}
    \caption{An illustration of a shifted plane partition of shape $\lambda/\mu=7521/31$ and the Hasse diagram of the associated poset $Q_{\lambda/\mu}$.}\label{fig:ex shifted plane partition}
\end{figure}

Based on calculations using {\tt SageMath} \cite{sagemath} for all shifted skew shapes of size $\leq 15$, it appears that the order polynomial of a shifted skew shape has nonnegative coefficients\footnote{In forthcoming work \cite{MMY}, Marberg, the second named, and Yu settled the case of shifted straight shapes using Schubert polynomials.}.

\begin{conj}\label{conj:shifted}
    The order polynomial of shifted skew shapes $\lambda/\mu$ has nonnegative coefficients, i.e., 
    $|\la/\mu|! \cdot \Omega(Q_{\la/\mu};t) \in \mathbb{Z}_{\geq 0}[t].$
\end{conj}

For shifted straight shapes, Hopkins and Lai \cite[Thm. 2.2]{HopkinsLai} used the theory of  Stembridge from \cite{StembridgePfaff}, building on work of Okada in \cite{OkadaPP}, to give a {\em Pfaffian} formula for $\ShPP_{\lambda}(t)$. This formula can be viewed as an analogue of MacMahon's determinant formula for $\PP_{\lambda}(t)$ (Theorem~\ref{Kreweras_form}). A Pfaffian formula for plane partitions of shifted skew shapes would make it easier to analyze the  linear term of $\Omega(Q_{\la/\mu};t)$.

\section{Final remarks}

\subsection{\texorpdfstring{$P$}{P}-Eulerian polynomials of skew shapes}

The numerator of the generating function of $\ehr(\mathscr{O}(P);t)$, commonly known as the $h^*$-polynomial of $\mathscr{O}(P)$, has nonnegative integer coefficients. This polynomial, denoted $h^*(P_{\lambda/\mu};z)$ is also referred to as the $P_{\lambda/\mu}$-Eulerian polynomial.

We do not know if there is a nice interpretation for the coefficients of $h^*(P_{\lambda/\mu};z)$. However, based on computer experiments we raise the following question.

\begin{Q}
    Is $h^*(P_{\lambda/\mu}; z)$ real-rooted for every skew shape $\lambda/\mu$?
\end{Q}

We note that, in the hierarchy of conjectures mentioned in \cite{ferroni-higashitani}, we \emph{expect} that $h^*(P_{\lambda/\mu};z)$ has log-concave coefficients. The reason is that $\mathscr{O}(P)$ possesses quadratic triangulations and hence has the integer decomposition property: no examples of IDP polytopes with not log-concave $h^*$-vectors are known. On the other hand, the log-concavity for $h^*(P_{\lambda/\mu}; z)$ would be a special case of the log-concavity version of the Neggers--Stanley conjecture due to Brenti \cite{BrentiNS}, see also \cite{xifan-yu} and \cite{alexandersson-jal}. When $\mu=0$ and $\lambda$ is a hook, the real-rootedness follows from \cite[Corollary~4.1]{ferroni_hooks}.

\subsection{Geometric proof of positivity for skew shapes}

In this paper we have established the positivity of the order polynomial of cell posets of skew shapes by a purely enumerative argument. However, the following remains as an open problem.

\begin{problem}
    Find a geometric proof of the nonnegativity of the coefficients of $\Omega(P_{\lambda/\mu};t)$.
\end{problem}

In particular, we raise the challenge of finding a half-open decomposition that explains the nonnegativity of the coefficients of $\Omega(P_{\lambda/\mu};t)$, perhaps also settling Conjecture~\ref{conj:expansion}, noting that half open products of simplices have Ehrhart polynomials that look like each summand in that conjecture.



\subsection{Real rootedness and interpretation of the coefficients of \texorpdfstring{$\Omega(P_{\lambda/\mu};t)$}{Omega(P_lambda/mu)(t)}}\label{ss:putnam} 

Problem B5 from the 2024 William Lowell Putnam Mathematical competition \cite{PutnamArchive}, inspired by the result in \cite{ferroni_hooks}, was to show that order polynomial of the poset corresponding to the straight shape of a hook $\lambda=(a+1,1^b)$ is a polynomial in $t$ with nonnegative coefficients. The solutions \cite{putnam} by the Putnam committee, who one of the members was the third named author, consist, in essence, of finding ad-hoc recursions which allow one to write explicit formulas for the polynomial showing positivity. The second anonymous solution of this problem in \cite{AltSolutionsPutnam2024} also shows that $\Omega(P_{(a+1,1^b)};t)$ is real-rooted. However, this does not hold for other (skew) shapes. For example, for $\lambda=332$, we have 
that 
\[
\Omega(P_{332};t) = 
\tfrac{1}{960}t^8 + \tfrac{1}{80}t^7 + \tfrac{91}{1440}t^6 + \tfrac{7}{40}t^5 + \tfrac{827}{2880} t^4 + \tfrac{67}{240} t^3 + \tfrac{107}{720} t^2 + \tfrac{1}{30}t,
\]
which is not real rooted.

In January 2025, Sam Hopkins asked in MathOverflow~\cite{hopkins} if there are generalizations to this result and interpretations for the coefficients. The present paper shows that positivity holds for all skew shapes. Explicit formulas for the coefficients in the straight shape case arise from Macdonald's identity for Schubert polynomials as explained in Section~\ref{sec:schubert}. In general, however, we lack explicit formulas and combinatorial interpretations for the coefficients \footnote{In \cite{Kahane}, Yakob Kahane  found a combinatorial interpretation for the case of (circular) fence posets and a conjectured interpretation for general skew shapes.}

\subsection{On  positivity of plane partitions of stretched skew shape}

Combining \eqref{eq:order-ehrhart} and \eqref{eq:omega of lam/mu is PP}, we know that the number $\PP_{\lambda/\mu}(t)$ for $t\in \mathbb{N}$ counts the number of lattice points of the dilated order polytope $t\cdot \mathscr{O}(P_{\lambda/\mu})$. There is another polytope whose lattice points are counted by $\PP_{\lambda/\mu}(t)$, the {\em generalized Pitman--Stanley polytope} defined in \cite[Section 5]{Pitman_Stanley_1999} and further studied in \cite{gen_PS1,gen_PS2} by Dugan, Hegarty, Morales and Raymond. This polytope is defined as follows. If we fix $n,t\in \mathbb{N}$ and given ${\bf a}=(a_1,\ldots,a_n), {\bf b}=(b_1,\ldots,b_n) \in \mathbb{Z}_{\geq 0}^n$, let 
\[
\mathscr{PS}_n^t({\bf a},{\bf b}) = \{ {\bf x} \in \mathbb{R}^{n\times t}_{\geq 0} \mid {\bf a} \trianglerighteq {\bf x_{\cdot 1}}\trianglerighteq {\bf x_{\cdot 1}} \trianglerighteq  \cdots \trianglerighteq  {\bf x_{\cdot t}} \trianglerighteq {\bf b} \},
\]
where $\trianglerighteq$ denotes {\em dominance order}, i.e. $(c_1,\ldots, c_n) \trianglerighteq (d_1,\ldots, d_n)$ if $\sum_{i=1}^j c_i \geq \sum_{i=1}^j d_i$ for $j=1,\ldots,n$, and ${\bf x_{\cdot j}}$ denotes the $j$th column of the matrix ${\bf x}$. When $t=1$, the polytope $\mathscr{PS}_n^t({\bf a},{\bf b})$ is the classical Pitman--Stanley polytope from \cite{Pitman_Stanley_1999}. The lattice points of $\mathscr{PS}_n^t({\bf a},{\bf b})$ are in natural correspondence with plane partitions of shape $\lambda({\bf a})/\mu({\bf b})=(a_1+\cdots +a_n,\cdots,a_1+a_2,a_1)/(b_1+\cdots +b_n,\cdots,b_1+b_2,b_1)$ with entries in $0,1,\ldots,t$ (see also \cite[Theorem 3.7]{gen_PS1}).  In this setting, dilating the polytope by $k \in \mathbb{N}$, corresponds to stretching the parts of $\lambda$ and $\mu$ by $k$ which we denote by $k\lambda$ and $k\mu$, respectively. Pitman and Stanley showed in \cite[Corollary 10]{Pitman_Stanley_1999} that for $t=1$ and ${\bf b}={\bf 0}$, the polytope $\mathscr{PS}_n^1({\bf a},{\bf 0})$ is Ehrhart positive. In in \cite{gen_PS2}, this positivity is conjectured to hold for $\mathscr{PS}_n^t({\bf a},{\bf b})$ as well.\footnote{After this paper was finished, this conjecture was settled by Jochemko and Menon \cite[Thm. 3.2]{JochemkoMenon}} 

\begin{conj}[\cite{gen_PS2}]
For fixed $t\in \mathbb{N}$, the  polynomial $\PP_{k\lambda/k\mu}(t)$ has nonnegative coefficients in $k$.    
\end{conj}

\begin{rem}
The polytope $\mathscr{PS}_n^t({\bf a},{\bf b})$ can be seen as a plane partition analogue of the {\em (skew) Gelfand--Tsetlin polytope} \cite{GT50} whose lattice points correspond to bounded semistandard Young tableaux of skew shape. In this case, the Ehrhart positivity for the straight shape case follows from the hook-content formula \cite[Cor.~7.21.4]{EC2}.
The Ehrhart positivity for the skew case is a conjecture of Alexandersson and  Alhajjar \cite[Conjecture 7]{AlexanderssonAlhajjar}\footnote{After this paper was finished, this conjecture was settled by  Jochemko and Menon \cite[Thm. 3.5]{JochemkoMenon}}.
\end{rem}

\subsection{Kahn--Saks monotonicity}
In~\cite[Exercise 3.163(b)]{ec1}, see also~\cite{chan2023effective}, Kahn--Saks conjectured that for every poset $P$ on $n$ elements, the map 
\[t \longmapsto \frac{\Omega(P;t)}{t^n}\]
is weakly decreasing when $t$ ranges over the positive integers. The nonnegativity of the coefficients of the order polynomial for cell posets $P_{\lambda/\mu}$ of skew shapes immediately gives that $\Omega(P_{\lambda/\mu};t)/t^n$ is a decreasing function of $t$ and hence Kahn--Saks monotonicity holds for these classes of posets.

\subsection{Linear terms and roots}

For a poset $P$ we have that 
\[c_1(P) = \frac{d \Omega(P;t)}{dt} \bigg{|}_{t=0} = \Omega'(P;0).\] 
We also have that $\Omega(P;0)=0$ and $\Omega(P;1)=1$. If $c_1(P)<0$ then $\Omega'(P;0)<0$, which says that $\Omega(P;t)$ is decreasing at $0$. Then, we must have that $\Omega(P;\varepsilon)<0$ for small positive values $\varepsilon \in (0,1)$. The mentioned condition $\Omega(P;1)>0$ says that $\Omega(P;t)$ must have an odd  number of roots $ \in (0,1)$ counting multiplicities. What are the possible number of roots of $\Omega(P;t)=0$ in the interval $t \in (0,1)$? 


\subsection{Non-examples}\label{ss:nonexamples}

As was mentioned in Example~\ref{ex:faulhaber}, one of the most basic counterexamples to Ehrhart positivity is the antichain on $n$ elements with a minimal (or maximal) element attached. For $n=4$, the resulting poset has an order polynomial with a negative coefficient. This implies that even simple operations like `adding a minimum' do not preserve the positivity of the order polynomial. Also, notice that this also says that the planarity of the Hasse diagram, or being a series-parallel poset, are not enough to guarantee positivity.

A natural family of posets that encompasses all cell posets of skew Young diagrams consists of all posets in which every element covers and is covered by at most $2$ elements. Unfortunately, that class also does not have a positive order polynomial, due to the example depicted in Figure~\ref{fig:2-covers}.

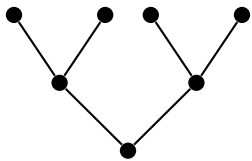
\begin{figure}[ht]
\begin{tikzpicture}
    [scale=0.6,auto=center,every node/.style={circle,scale=0.8, fill=black, inner sep=2.7pt}] 
    \tikzstyle{edges} = [thick];
    	
    \node[] (a) at (0,5) {};
    \node[]  (b) at (2,5) {};
    \node[] (e) at (1,3.5) {};
    \node[] (c) at (3,5) {};
    \node[] (d) at (5,5) {};
    \node[] (g) at (4,3.5) {};
    \node[] (h) at (2.5,2) {};
    
    \draw[edges] (a)--(e);
    \draw[edges] (b)--(e);
    \draw[edges] (c)--(g);
    \draw[edges] (d)--(g);
    \draw[edges] (e)--(h);
    \draw[edges] (g)--(h);
\end{tikzpicture}\caption{Every element covers and is covered by at most $2$ elements.}\label{fig:2-covers}
\end{figure}

A direct computation shows that 
\[\Omega(P,t) = \tfrac{1}{63} t^{7} + \tfrac{1}{9} t^{6} + \tfrac{53}{180} t^{5} + \tfrac{13}{36} t^{4} + \tfrac{7}{36} t^{3} + \tfrac{1}{36} t^{2} - \tfrac{1}{210} t,\]
which has a negative linear term. Moreover, this is the smallest example of this kind in which there is a negative coefficient in the order polynomial. This explains that even though $c_1(P)<0$, the other coefficients are indeed nonnegative.

\bibliographystyle{amsalpha}
\bibliography{bibliography_fences}

\end{document}